\documentclass[12pt,draftcls,onecolumn]{IEEEtran}
\usepackage{algorithmic,algorithm}
\usepackage{amssymb}
\usepackage{amsmath,verbatim}
\usepackage{amsthm}
\usepackage{rotating}
\usepackage{caption}
\usepackage{mathrsfs}
\usepackage{graphicx}
\usepackage{hyperref}
\usepackage{array}
\usepackage{float}
\usepackage{color}
\usepackage{comment}
\usepackage{calc}
\usepackage{cite}
\usepackage{comment}
\usepackage{multicol}
\newcommand{\subparagraph}{}
\usepackage[compact]{titlesec}

\newtheorem{assumption}{Assumption}
\newtheorem{thm}{Theorem}
\newtheorem{lemma}{Lemma}

\newtheorem{pre}{Proposition}
\newtheorem{rem}{Remark}
\newtheorem{com}{Comparison}

\newtheoremstyle{newstyle}     
{0 pt} %Aboveskip 
{0 pt} %Below skip
{\itshape} %Body font e.g.\mdseries,\bfseries,\scshape,\itshape
{} %Indent
{\bfseries} %Head font e.g.\bfseries,\scshape,\itshape
{.} %Punctuation afer theorem header
{.5em} %Space after theorem header
{} %Heading
\theoremstyle{newstyle}
\newcommand{\EXP}[1]{\mathsf{E}\!\left[#1\right] } 

\newcommand{\nm}[1]{{\color{black}#1}}
\newcommand{\fy}[1]{{\color{black}#1}}
\newcommand{\ap}[1]{{\color{black}#1}}

\usepackage[compact]{titlesec}  
\titlespacing{\section}{0pt}{2pt}{0pt}

\author{Nahidsadat~Majlesinasab,   %~\IEEEmembership{Member,~IEEE,}
        \and Farzad~Yousefian,   %~\IEEEmembership{Fellow,~OSA,} 
		\and Arash~Pourhabib\thanks{The authors are with the School of Industrial Engineering and Management, Oklahoma State University, OK 74074, USA. They are contactable at 
\{nahid.majlesinasab, farzad.yousefian, arash.pourhabib\}@okstate.edu.}}

\title{Self-tuned mirror descent schemes for smooth and nonsmooth high-dimensional stochastic optimization}
\begin{document}
\maketitle
\thispagestyle{empty}
\pagestyle{plain}

\begin{abstract}
We consider randomized block coordinate stochastic mirror descent (RBSMD) methods for solving high-dimensional stochastic optimization problems with strongly convex objective functions. Our goal is to develop RBSMD schemes that achieve a rate of convergence with a minimum constant factor with respect to the choice of the stepsize sequence. To this end, we consider \fy{both subgradient and gradient RBSMD methods addressing nonsmooth and \ap{smooth} problems, respectively.} For each scheme, \fy{(i) }we develop self-tuned stepsize \fy{rules characterized} in terms of problem parameters and algorithm settings\fy{; (ii) we show} that the non-averaging iterate generated by the underlying RBSMD method converges to the optimal solution both in an almost sure and a mean sense; (iii) \ap{we show that  the mean squared error is minimized}. \fy{When} problem parameters are unknown, we develop a unifying self-tuned update rule that can be applied in both \fy{subgradient and gradient SMD methods, and show that for any arbitrary and small enough initial stepsize, a suitably defined error bound is minimized.} We provide constant factor comparisons with standard \fy{SMD and RBSMD} methods. Our numerical experiments performed on an SVM model display that the \fy{self-tuned schemes} are significantly robust with respect to the choice of problem parameters, and \fy{the }initial stepsize.
\end{abstract}

%{\bf Keywords}: stochastic optimization, mirror descent, self-tuned stepsize, $\ell_1$ regularization, averaging
\section{Introduction}\label{sec_intro}
%\section{Problem formulation}\label{sec:prob}
In this work, we consider the canonical stochastic optimization problem given by 
\begin{align}
\label{eqn:problem1} \tag{SO}
\underset{\beta \in \mathscr B} {\text{minimize}}\quad F(\beta):=\EXP{f(\beta,\xi)},
\end{align}
where $\mathscr B\triangleq \prod_{i=1}^l \mathscr B_i \subset \mathbb{R}^n$ \fy{with $\mathscr B_i \subset\mathbb R^{n_i}$ being nonempty, closed, and convex sets, }$n\triangleq\sum_{i=1}^l n_i$, and the function $f:\mathscr B\times\Omega \to \mathbb{R}$ is a stochastic function.
The vector $\xi:\Omega \to \mathbb{R}^d$ is a random vector associated with a probability space represented by $(\Omega, \mathcal{F},\mathbb{P})$. A wide range of problems in machine learning and signal processing can be formulated as problem \eqref{eqn:problem1}. 
In addressing problem \eqref{eqn:problem1}, challenges arise in the development of efficient solution methods \fy{mainly due to:} (i) presence of uncertainty: in many applications arising in statistical learning, the probability distribution $\mathbb{P}$ is unknown. Even when \fy{$\mathbb{P}$} is known, the evaluation of the expectation \fy{of $f$} becomes costly, \ap{especially} when $d >5$; (ii) high-dimensionality \fy{of the solution space}: another difficulty arises \fy{when $n$} is huge. In such applications \ap{the computational complexity of first-order methods significantly increases, which makes them impractical} for large values of $n$ (e.g., $10^{12}$ or more). In addressing uncertainty, \fy{the} stochastic approximation (SA) method was \ap{first} developed by Robbins and \fy{Monro \cite{robbins1951stochastic}.} Since then, \fy{the }SA method and its variants have been vastly employed \fy{in addressing} stochastic optimization (\hspace{1sp}\cite{nevelson1973stochastic,ermoliev1983stochastic,kushner2003stochastic}). Averaging techniques first introduced by Polyak and Juditsky \cite{Polyak92} proved successful in increasing the robustness of \ap{the} SA method. In vector spaces equipped with non-Euclidean norms, prox generalizations of deterministic gradient method (\hspace{1sp}\cite{yudin1983problem,beck2003mirror}) were introduced and applied in smooth and nonsmooth \fy{regimes}. Also, in \ap{the} stochastic regime, Nemirovski et al. \cite{nemirovski2009robust} developed the stochastic mirror descent (SMD) method for solving problem \eqref{eqn:problem1} when $F$ is nonsmooth and convex. Under a window-based averaging scheme, the rate of $\mathcal{O}\left(1/\sqrt{t}\right)$ is established.
When the \ap{dimension} of the solution space is huge\fy{, to reduce the} computational burden, \fy{coordinate \ap{descent}} (CD) methods have been developed in recent decades \fy{(cf. \cite{ortega1970iterative,nest10,Wotao13})}. Recently, Dang and Lan \cite{dang2015stochastic} developed \fy{the} stochastic block mirror descent (SBMD) method. They showed that \fy{under averaging} the convergence rate of $\mathcal{O}\left(l/\sqrt{t}\right)$ and $\mathcal{O}\left(l/t\right)$ can be established for the case when $F$ is merely convex, and strongly convex, respectively. 
Some recent works have considered SMD on a class of non-convex \fy{problems called \textit{variationally coherent }}\ap{that} includes pseudo-/quasiconvex\fy{,}
and star-convex optimization problems and provided convergence rate of $\mathcal{O}\left(1/\sqrt{t}\right)$ for strongly coherent problems \cite{zhou2017stochastic,zhou2017mirror}.
 While these non-asymptotic convergence orders are known to be optimal for the SMD method, the performance of this method can be significantly sensitive with respect to \fy{problem parameters and the uncertainty}. Much of the interest in the literature has focused on establishing the optimal convergence rates, and there is little guidance on development of stepsize update rules for the SMD method in order to minimize the constant factor of the associated error bounds. Motivated by this gap, our goal in this work lies in improvement of the finite-time behavior of the SMD methods and their block variants through development self-tuned stepsizes. Several efforts have been done in development of efficient stepsize rules for SA schemes (e.g., see \cite{kesten1958accelerated,saridis1970learning}). Spall \cite{Spall03} discusses a harmonic stepsize of the form $\eta_t=\frac{a}{(t+1+A)^\alpha}$ where $a>0$ is a tuning parameter, and $A \geq 0$ is the stability constant. George and Powell \cite{george2006adaptive} propose a class of harmonic stepsizes which minimizes the mean squared estimation error \fy{(see also \cite{benveniste1990stochastic}, \cite{pflug1988stepsize}, \cite{stengel1994optimal}, \cite{darken1992towards})}. Self-tuned stepsizes were first introduced in \cite{Farzad1} where a recursive update rule is developed for the stochastic (sub)gradient methods. It is shown that using such update rules, the mean squared error of the method is minimized w.r.t. the stepsize choice. Motivated by that work, we develop self-tuned stepsize rules for SMD methods and their randomized block variants (RBSMD) to solve problem \eqref{eqn:problem1} where the objective function $F$ is strongly convex. We consider two cases where (i) $F$ is nondifferentiable, (ii) $F$ is differentiable and has Lipschitz gradients. 
%While the SMD methods developed in the literature (cf. \cite{nemirovski2009robust,dang2015stochastic}) employ averaging, our goal lies in developing non-averaging schemes. 
\fy{In the following, we present the main contributions and explain the distinctions with the earlier works \cite{Farzad1,Farzad2}}: \\
\noindent \textit{(1) Convergence and complexity analysis:} For each variant of the RBSMD methods, we develop new recursive error bounds. These error bounds are given by Lemma \ref{lm:bsmd} and \ref{lem:ineqSmooth} for cases (i) and (ii), respectively. In each case, we then develop self-tuned stepiszes that are characterized in terms of problem parameters and algorithm settings. We show that under such update rules, the error function of the underlying RBSMD method converges to zero in an almost sure and a mean sense. Importantly, we show that the expected value of the error is minimized under the self-tuned stepize rules within a specified range. We also derive bounds on the probability of error of the RBSMD schemes in terms of problem parameters, algorithm settings, and \ap{the} iteration number (see Propositions \ref{pre:3} and \ref{pre:2} for cases (i) and (ii), respectively). Our results in this work extend the previous findings on self-tuned stepsizes in \cite{Farzad1,Farzad2} to a broader class of algorithms\fy{,} i.e., SMD methods and their \fy{randomized} block variants. Moreover, our approach in addressing nonsmoothness is different \ap{from} that considered in \cite{Farzad1,Farzad2}. Here we develop subgradient variants of \fy{the} RBSMD method allowing us to prove convergence to \fy{the true} optimal solution to problem \eqref{eqn:problem1}, while in \cite{Farzad1,Farzad2} a \fy{convolution-based }smoothing scheme is applied \fy{where the nondifferentiable function $F$ is approximated by a smooth function defined on an expansion of the feasible set $\mathscr{B}$. Consequently, in addressing nonsmooth problems, the convergence in \cite{Farzad1,Farzad2} is established to the optimal solution of an approximate smooth problem.}  
\\
\noindent \textit{(2) Unifying self-tuned stepsizes:} Another improvement to \cite{Farzad1,Farzad2} \fy{ in deriving self-tuned schemes is addressing the case where some of problem parameters are unknown. In this case,} we develop \textit{unifying self-tuned stepsizes} and \fy{show} convergence in both an almost sure and a mean sense. Importantly, we show that for \fy{any} arbitrary and small enough initial stepsize, a suitably defined error bound of the \fy{SMD scheme} is minimized (see Theorem \ref{thm:1}). This indeed implies robustness of the proposed schemes w.r.t. the choice of initial stepize and addresses a common challenge associated with the harmonic choice of stepsizes. \\
\noindent \textit{(3) Constant factor comparison:} While we prove the superiority of the constant factor of the error bounds associated with the \fy{self-tuned RBSMD schemes}, we also provide two sets of comparisons: (i) with a widely used harmonic stepsizes (e.g., in \cite{nemirovski2009robust,Spall03}), and also (ii) with an averaging RBSMD scheme developed in \cite{dang2015stochastic}. In case (ii), our comparison implies the constant factor for the class of stochastic subgradient methods can be improved up to four times under non-averaging schemes versus using the averaging scheme in\cite{dang2015stochastic}. 
%\noindent \textit{(4) Implementation results:} We present the performance of the unifying self-tuned stepsizes applied on SVM models under two different data sets. Our results indicate the robustness of the developed schemes with respect to problem parameters, uncertainty, and the initial stepsize. 

The rest of the technical note is organized as follows. In Section \ref{sec:bsmd}, we describe the randomized block coordinate SMD methods, then develop a self-tuned stepsize scheme for smooth and nonsmooth problems and provide the convergence rate analysis. We present \ap{the}  experimental results in Section \ref{sec:num} and conclude with some remarks in Section \ref{sec:conclusion}.

%\section{Related Work}\label{sec_lit}
\textbf{Notation}. Throughout, we abbreviate ``almost surely'' as a.s., while $Prob~(Z)$ and $\EXP{z}$ are used to denote the probability of an event $Z$, and the expectation of a random variable $z$, respectively. We let $\beta^i \in \mathbb R^{n_i}$ \fy{denote} the $i$th block coordinate of vector $\beta \in \mathbb R^n$, and the subscript $i$ \fy{represent} the $i$th block of a mapping in $\mathbb R^n$. For any $i=1,\ldots,l$, we use $\| \cdot \|_i$ to denote the general norm on $\mathbb R^{n_i}$ and $\| \cdot \|_{*i}$ to denote its dual norm. The inner product of vectors $u, v \in \mathbb R^{n}$ is  defined by $\left\langle u, v\right\rangle:\triangleq \sum_{i=1}^l\left\langle u^i, v^i\right\rangle$. We define norm $\|\cdot\|$ as $\|x\|^2:\triangleq\sum_{i=1}^d\|x^i\|_i^2$ for any $x \in \mathbb R^n$, and denote its dual norm by$\|\cdot\|_*$. \fy{Throughout, $p_i$ denotes the probability associated with choosing the $i$th block coordinate.}
We \fy{use the notation} $p_\wedge:\triangleq\min\limits_{1\leq i\leq l}p_i$, $p_\vee:\triangleq\max\limits_{1\leq i\leq l}p_i$, $\mathfrak L_{max}:\triangleq\max\limits_{1\leq i\leq l}\mathfrak L_{\omega_i}$, and $\mu_{min}:\triangleq\min\limits_{1\leq i\leq l}\mu_{\omega_i}$.

\section{Self-tuned randomized block coordinate SMD}
\label{sec:bsmd}
{
In this section, our goal is to develop \fy{self-tuned randomized block coordinate SMD} methods. We start with the case where the
objective function is non-differentiable. Later, in Section \ref{subsec:diff}, we discuss the case of differentiable objective functions with
Lipschitz gradients. In Section \ref{subsec:unifying}, we provide unifying self-tuned update rules addressing both cases in absence of problem
parameters.
Let the distance generating function \fy{$\omega_i:\mathbb{R}^{n_i} \to \mathbb{R}$} be a continuously differentiable function. The Bregman divergence \fy{$D_{\omega_i}:\mathbb{R}^{n_i}\times \mathbb{R}^{n_i} \to \mathbb{R}$} associated with $\omega_i$ \fy{is given for $\beta_1, \beta_2 \in \mathscr B_i$ as}
\begin{equation*}
D_{\omega_i}(\beta_1, \beta_2)=\omega_i(\beta_2)-\omega_i(\beta_1)-\langle\nabla\omega_i(\beta_1),\beta_2-\beta_1\rangle.
\end{equation*}
Let $\nabla _{\beta_2} D_{\omega_i}(\cdot,\cdot)$ denote the partial derivative of $D_{\omega_i}(\beta_1, \beta_2)$  with respect to  $\beta_2$. Then, 
\begin{equation}
\label{eq:73}
\nabla_{\beta_2} D_{\omega_i}(\beta_1,\beta_2)=\nabla \omega_i(\beta_2)-\nabla \omega_i(\beta_1), \quad \hbox{for all} \quad \beta_1, \beta_2 \in \mathscr B_i.
\end{equation}
The Bregman divergence has the following property for all $\beta_1, \beta_2,\beta_3 \in \fy{\mathscr B_i}$
\begin{align}
\label{eq:74}
&D_{\omega_i}(\beta_1, \beta_2)-D_{\omega_i}(\beta_3, \beta_2)=D_{\omega_i}(\beta_1, \beta_3)+\langle\nabla\omega_i(\beta_3)-\nabla\omega_i(\beta_1),\beta_2-\beta_3\rangle.
\end{align}
We assume the distance generating function $\omega_i$ has Lipschitz gradients with parameter $\mathfrak L_{\omega_i}$ and is strongly convex with parameter $\mu_{\omega_i}$, i.e., for all $\beta_1, \beta_2,\beta_3 \in \fy{\mathscr B_i}$ 
\begin{equation}
\label{eq:75}
\frac{\mu_{\omega_i}}{2}\|\beta_2-\beta_1\|^2 \leq D_{\omega_i}(\beta_1, \beta_2) \leq \frac{\mathfrak L_{\omega_i}}{2}\|\beta_2-\beta_1\|^2.
\end{equation}
}
\begin{rem}
Lipschitzian property of \fy{$\omega_i$} is a standard assumption in the literature of SMD methods; the convergence \fy{rate} analysis provided \fy{in \cite{nedic2014stochastic} and \cite{dang2015stochastic}} relies on this property. \fy{Also note that for the} stochastic gradient descent (SGD) \fy{method}, we have $\mu_\omega=\mathfrak L_\omega=1$.
\end{rem}
The prox mapping $\mathcal P_i:\mathscr B_i\times\mathbb R^{n_i}\rightarrow \mathscr B_i$ is defined by
\begin{align}
\label{eq:proxi}
\mathcal P_i(\beta_1,\beta_2)=\underset{z \in \mathscr B_i}{\text{argmin}}\{\left\langle \beta_2,z\right\rangle+D_i(\beta_1,z)\},
\end{align}
for all $\beta_1 \in \mathscr B_i$ \fy{and} $\beta_2 \in \mathbb R^{n_i}$. \fy{In the analysis, }we use the following error function $\mathcal L:\mathscr B\times \mathscr B\rightarrow\mathbb R$ defined as 
\begin{align}
\label{eq:lypanov}
\mathcal L (\beta, z)\triangleq \sum\nolimits_{i=1}^l {p_i}^{-1}{D_i(\beta^i, z^i)},\quad  \hbox{for all } \beta, z \in \mathscr B.
\end{align}
\subsection{Self-tuned randomized block subgradient SMD method}\label{subsec:nondiff}
Consider problem \eqref{eqn:problem1} where $F$ is a non-differentiable convex function of $\beta$. Let $\mathbf g_t \in \partial F(\beta_t)$ denote a subgradient of function $F$ at point $\beta_t\in \mathscr B$, \fy{i.e., there exists $\mathbf g_t$ such that} 
\begin{align*}
F(\beta_t)+\langle\mathbf g_t, \beta- \beta_t\rangle \leq F(\beta), \quad \hbox{for all } \beta \in \mathscr B.
\end{align*}
Similarly, for any $\xi \in \Omega$, we let $\tilde g_t \in \partial f(\beta_t,\xi)$ denote a subgradient of function $f(\cdot,\xi)$ at point $\beta_t$. \fy{ Throughout,} we assume that $F$ is strongly convex with parameter $\mu_F>0$ over the set $\mathscr{B}$ with respect to the underlying norm $\|\cdot\|$, i.e., for all $\beta_1, \beta_2 \in \mathscr B$ and $\mathbf g \in \partial F(\beta_2)$
\begin{equation}\label{ineq:strong} 
F(\beta_1) \geq F(\beta_2)+ \langle \mathbf g,\beta_1-\beta_2\rangle+\frac{\mu_F}{2}\|\beta_1-\beta_2\|^2. 
\end{equation}
In our analysis, we make use of the following result. The proof can be found in \cite{nesterov2013introductory}.
\begin{lemma}\label{lm:2}
Consider problem \eqref{eqn:problem1}. Let $F$ be strongly convex with parameter $\mu_F>0$. Then, there exists a unique optimal solution $\beta^* \in \mathscr B$. Moreover, we have 
\begin{equation*} 
F(\beta)-F(\beta^*) \geq \frac{\mu_F}{2}\|\beta-\beta^*\|^2, \quad \hbox{for all } \beta \in \mathscr B.
\end{equation*}
\end{lemma}
Next we present the outline of the randomized block coordinate SMD method. Let $P_b$ be a discrete probability distribution with probabilities $p_i>0$ for $i=1,\ldots,l$, where $\sum_{i=1}^lp_i=1$. Given an initial vector $\beta_0 \in \mathscr B$, at iteration $t\geq 1$, random variable $i_t$ is generated from the probability distribution $P_b$ independently from random variable $\xi$. Then, only the $i_t$th block of $\beta_t$, i.e. $\beta_t^{i_t}$, is updated as follows:
				 \begin{equation}\label{alg:BSSMD}\tag{RB-SSMD}
					\beta_{t+1}^{i}=\left\{
					\begin{array}{ll}
           \mathcal P_{i_t}\big(\beta_{t}^{i_t}, \eta_t \tilde g_{i_t}(\beta_t)\big) & \text{if } i = i_t,\\
           \beta_{t}^{i}& \text{if } i \neq i_t,
     \end{array} \right.
				 \end{equation} 
where $\tilde g_{i_t}(\beta_t)$ is the $i_t$th block of the \fy{subgradient of $f(\beta_t,\xi_t)$ and $\eta_t$ is the \nm{stepsize}. Throughout, let $\mathcal F_t=\{i_0,\xi_0,\ldots,i_{t-1},\xi_{t-1}\}$.} %We define $z_t$ as the difference between the sample gradient $\tilde g(\beta)$ and its expectation $\mathbf g(\beta)$ estimated at $\beta=\beta_t$, i.e. $z_t=\tilde g(\beta_t)-\textbf g(\beta_t)$.
 Next, we state the main assumptions.
\begin{assumption}
\label{ass:3}
Let the stochastic subgradient $\tilde g(\beta) \in \partial f(\beta,\xi)$ be such that a.s. for all $\beta \in \mathscr B$, we have $\EXP{\tilde g(\beta)|\beta}=\mathbf g(\beta)  \in \partial F(\beta)$. Moreover, for all $i=1,\ldots, l$ and $\beta \in \mathscr B$, there exists a scalar $C_i>0$ such that $
\EXP{\| \tilde g_i(\beta)\|_{*_i}^2|\beta}\leq  C_i^2$.
\end{assumption}
Next, we develop a recursive inequality in terms of the error of the \eqref{alg:BSSMD} scheme. Such a recursive inequality will be employed to develop a self-tuned stepsize rule.
\begin{lemma}\label{lm:bsmd}
Let Assumption \ref{ass:3} hold and $\beta_t$ be generated by the \eqref{alg:BSSMD} scheme. Then for all $t\geq 0$,
\begin{align}
\label{eq:bsmd_ineq}
\EXP{\mathcal L (\beta_{t+1}, \beta^*)|\mathcal F_t} &\leq\left(1-\eta_t{2\mu_F p_\wedge}{\mathfrak L^{-1}_{max}} \right)\mathcal L (\beta_{t}, \beta^*)+\eta_t^2 \sum\nolimits_{i=1}^l C_{i}^2({2\mu_{\omega_i}})^{-1}.
\end{align}
\end{lemma}
\begin{proof}
%For notational convenience, we denote 
At iteration $t$, we have $\beta_{t+1}^{i_t}=\mathcal P_{i_t}\big(\beta_{t}^{i_t}, \eta_t \tilde g_{i_t}(\beta_t)\big)$. Consider the definition of $\mathcal P_{i_t}$ given by \eqref{eq:proxi}. Writing the optimality \fy{condition}, we have
\begin{equation*}%\label{eq:202}
\langle{\eta_t} \tilde g_{i_t} +\nabla D_{i_t}(\beta_t^{i_t}, \beta_{t+1}^{i_t}), \beta^{i_t}-\beta_{t+1}^{i_t}\rangle \geq 0, \quad \hbox{for all } \beta \in \mathscr B.
\end{equation*}
Using relations (\ref{eq:73}) and (\ref{eq:74}), and from the preceding relation,
\begin{align*}
&D_{i_t}(\beta_t^{i_t}, \beta^{i_t})-D_{i_t}(\beta_{t+1}^{i_t}, \beta^{i_t})-D_{i_t}(\beta_t^{i_t}, \beta_{t+1}^{i_t})\geq {\eta_t}\langle \tilde g_{i_t} ,\beta_{t+1}^{i_t}-\beta^{i_t}\rangle\fy{, \quad \text{for all } \beta \in} \mathscr B.
\end{align*}
From the strong convexity of $\omega_{i_t}$ and relation~(\ref{eq:75}), we have
\begin{align}
\label{eq:77}
&D_{i_t}(\beta_t^{i_t}, \beta^{i_t})-D_{i_t}(\beta_{t+1}^{i_t}, \beta^{i_t})-0.5{\mu_{\omega_{i_t}}}\|\beta_t^{i_t}-\beta_{t+1}^{i_t}\|_{i_t}^2 \geq{\eta_t}\langle \tilde g_{i_t} ,\beta_{t+1}^{i_t}-\beta^{i_t}\rangle.
\end{align}
By adding and subtracting ${\eta_t}\langle \tilde g_{i_t} ,\beta_{t}^{i_t}\rangle$ in the right-hand side, \fy{and using Fenchel's inequality, we have}
\begin{align}
\label{eq:78}
{\eta_t}\langle \tilde g_{i_t} ,\beta_{t+1}^{i_t}-\beta_t^{i_t}\rangle+{\eta_t}\langle \tilde g_{i_t} ,\beta_{t}^{i_t}-\beta^{i_t}\rangle 
%& \geq -\Big|\langle \frac{\eta_t}{\sqrt{\mu_\omega}}\tilde g_t ,\sqrt{\mu_\omega}(\beta_{t+1}-\beta_t)\rangle\Big|+{\eta_t}\langle \tilde g_t ,\beta_{t}-\beta\rangle \nonumber\\
&\geq  {-0.5\eta_t^2}{\mu^{-1}_{\omega_{i_t}}}\| \tilde g_{i_t}  \|^2_{*i_t}-0.5{\mu_{\omega_{i_t}}}\|\beta_{t+1}^{i_t}-\beta_{t}^{i_t}\|_{i_t}^2\notag\\
&+{\eta_t}\langle \tilde g_{i_t} ,\beta_{t}^{i_t}-\beta^{i_t}\rangle.
\end{align}
Combining~(\ref{eq:77}) and~(\ref{eq:78}) yields for all $\beta \in \mathscr B$
\begin{align*}
D_{i_t}(\beta_{t+1}^{i_t}, \beta^{i_t}) &\leq D_{i_t}(\beta_{t}^{i_t}, \beta^{i_t})+{\eta_t}\langle \tilde g_{i_t} ,\beta^{i_t}-\beta_{t}^{i_t}\rangle+0.5{\eta_t^2}{\mu^{-1}_{\omega_{i_t}}}\| \tilde g_{i_t}  \|^2_{*i_t}.
\end{align*}
%By taking the conditional expectation on $\mathcal F_t \cup \{i_t\}$ from both sides of the preceding relation, and using Assumption~\ref{ass:3}, for $\beta:=\beta^*$ we have
%\begin{align}\label{eq:87}
%{\eta_t}\langle \mathbf g_{i_t} ,\beta_{t}^{i_t}-\beta^{i_t}\rangle+\EXP{D_{i_t}(\beta_{t+1}^{i_t}, \beta^{i_t})\mid \mathcal F_t \cup \{i_t\}} \leq D_{i_t}(\beta_{t}^{i_t}, \beta^{i_t})+\eta_t^2\frac{C_{i_t}^2}{2\mu_{\omega_{i_t}}}.
%\end{align}
From the preceding relation, relation \eqref{eq:lypanov}, and that $\beta_{t+1}^{i}=\beta_t^i$ for all $i\neq i_t$, we have 
\begin{align*}
\mathcal L(\beta_{t+1}, \beta)&\leq\sum\nolimits_{i\neq i_t} p_i^{-1}D_i(\beta_t^i, \beta^i)+ p_{i_t}^{-1}\big(D_{i_t}(\beta_{t}^{i_t},\beta^{i_t})+\eta_t\langle \tilde g_{i_t},\beta^{i_t}-\beta_t^{i_t}\rangle+0.5{\eta_t^2}{\mu^{-1}_{\omega_{i_t}}}\| \tilde g_{i_t}  \|^2_{*i_t}\big)\\&=\mathcal L (\beta_{t}, \beta)
+p_{i_t}^{-1}\left(\eta_t\langle \tilde g_{i_t},\beta^{i_t}-\beta_t^{i_t}\rangle+0.5{\eta_t^2}{\mu^{-1}_{\omega_{i_t}}}\| \tilde g_{i_t}  \|^2_{*i_t}\right).
\end{align*}
Taking conditional expectations from both sides of the preceding relation on $\mathcal F_t\cup \{i_t\}$, we get 
\begin{align*}
\label{eq:bala}
\EXP{\mathcal L (\beta_{t+1}, \beta)\mid \mathcal F_t\cup \{i_t\}} &\leq \mathcal L (\beta_{t}, \beta)+0.5{\eta_t^2}{\mu^{-1}_{\omega_{i_t}}p^{-1}_{i_t}}\EXP{\| \tilde g_{i_t}  \|^2_{*i_t}\mid \mathcal F_t\cup \{i_t\}}\notag \\&+\frac{\eta_t}{p_{i_t}}\langle\EXP{\tilde g_{i_t}\mid \mathcal F_t\cup \{i_t\}},\beta^{i_t}-\beta_t^{i_t}\rangle \notag \\
&\leq\mathcal L (\beta_{t}, \beta)+p_{i_t}^{-1}\eta_t\left\langle \mathbf g_{i_t},\beta^{i_t}-\beta_t^{i_t}\right\rangle+p_{i_t}^{-1}\eta_t^2\frac{C_{i_t}^2}{2\mu_{\omega_{i_t}}},
\end{align*} 
where we \fy{used }Assumption \ref{ass:3}. Taking expectations from previous inequality with respect to $i_t$ and setting $\beta:=\beta^*$,
\begin{align*}
\EXP{\mathcal L (\beta_{t+1}, \beta^*)\mid \mathcal F_t} &\leq \mathcal L (\beta_{t}, \beta^*)+\sum_{i=1}^l\frac{p_i}{p_i}\left(\eta_t\left\langle \mathbf g_{i},\beta^{*i}-\beta_t^{i}\right\rangle +\eta_t^2\frac{C_{i}^2}{2\mu_{\omega_{i}}}\right)
\notag\\
&=\mathcal L (\beta_{t}, \beta^*)+\eta_t\left\langle \mathbf g_t,\beta^*-\beta_t\right\rangle+\eta_t^2\sum_{i=1}^l\frac{C_{i}^2}{2\mu_{\omega_{i}}},
\end{align*} 
where we use the definition of $\left\langle \cdot, \cdot\right\rangle$ given in the notation. From strong convexity of function $F$, we have
$\langle \mathbf g_t -\mathbf g^*,\beta_t -\beta^*\rangle \geq \mu_F\|\beta_t-\beta^*\|^2.$
By optimality of $\beta^*$, we have $\langle \mathbf g^*, \beta_t-\beta^*\rangle \geq 0$. From the two preceding relations and the definition of norm,
\begin{align*}
\left\langle \mathbf g_t,\beta_t-\beta^*\right\rangle &\geq \mu_F\sum_{i=1}^l\|\beta_t^i-\beta^{*i}\|_i^2\geq 2\mu_F\sum_{i=1}^l\frac{D_i(\beta_t^i,\beta^{*i})}{\mathfrak L_{\omega_i}}\\&\geq {2\mu_Fp_\wedge}{\mathfrak L^{-1}_{max}}\sum_{i=1}^l p_i^{-1}D_i(\beta_t^i,\beta^{*i})
={2\mu_Fp_\wedge}{\mathfrak L^{-1}_{max}}\mathcal{L}(\beta_t,\beta^{*}),
\end{align*} 
where in the second inequality we used relation \eqref{eq:75}, and in the last relation we used the definition of function $\mathcal L$. From \fy{the preceding two relations}, we obtain the desired inequality.
\end{proof}
 Next, we present some important properties of the self-tuned sequences. Some of them can be found in~\cite{Farzad1} and \cite{Farzad2}. The proof of the next lemma is presented in the Appendix. 
\begin{lemma}
\label{lemma:selftunedGeneral}
Let $\theta, \delta>0$ be scalars, and $\{e_t\}\geq 0$ be a sequence for $t \geq 0$, such that for an arbitrary non-negative sequence $\{\eta_t\}$
\begin{equation}
\label{eqn:seqer}
e_{t+1} := (1-\theta\eta_t) e_t+\delta\eta_t^2, \quad \hbox{for all } t\geq 1.
\end{equation}
Let $e_0\leq\frac{2\delta}{\theta^2}$ and let the self-tuned sequence $\{\eta_t^*\}$ be given by $\eta_{t}^*=\eta_{t-1}^*\left(1-\frac{\theta}{2}\eta_{t-1}^*\right)$ for any $t\geq1$, where $\eta_0^*=\frac{\theta}{2\delta}e_0$. Then the following properties hold:\begin{itemize}
\item [(a)] For any fixed $t \geq 1$, the vector $(\eta_0^*, \ldots,\eta_{t-1}^*)$ minimizes the function $e_t(\eta_0,\ldots,\eta_{t-1})$ over the set $\mathbb{U}_t\triangleq\bigg\{\gamma\in \mathbb{R}^t :0<\gamma_j\leq {1}/{\theta} ~~\text{for}~ j=1,\dots,t \bigg\}.$\\
More precisely, for $t\geq 1$, and any $(\eta_0,\ldots,\eta_{t-1}) \in \mathbb U_t$, 
\begin{align*}
e_t(\eta_0,\ldots,\eta_{t-1})-e_t(\eta_0^*,\ldots,\eta_{t-1}^*)\geq \delta(\eta_{t-1}-{\eta^*_{t-1}})^2.
\end{align*}
\item [(b)]  For all $t\geq 1$, we have 
 $\eta_t^*<\frac{2}{\theta }\left(\frac{1}{t}\right)$. Moreover, under the choice of $\eta_t:=\eta_t^*$, the term $e_t$ is bounded by $\mathcal O(1/t)$:
\begin{equation} 
\label{eq:59}
e_t(\eta_0^*, \eta_1^*,\ldots,\eta_{t-1}^*) \leq\frac{4\delta} {\theta^2}\left(\frac{1}{t}\right), \quad \hbox{for all }t \geq 1.
\end{equation}
\item [(c)]  We have $\sum_{t=0}^\infty\eta^{*}_t=\infty$ and $\sum_{t=0}^\infty\eta^{*2}_t<\infty$.
\end{itemize}
\end{lemma}
We use the following lemma \fy{in the convergence analysis.}
\begin{lemma}
\label{lm:6}\fy{[\cite{polyak1987introduction}, page 49]}
Let $\{v_t\}$ be a sequence of non-negative random variables where $\EXP{v_0}<\infty$, let $\{\alpha_t\}$ and $\{\lambda_t\}$ be deterministic scalar sequences such that
\begin{align*}
&\EXP{v_{t+1}|v_0,\ldots, v_t} \leq (1-\alpha_t)v_t + \lambda_t,\quad  \hbox{a.s. for all } t\geq 0, \\
&0 \leq \alpha_t \leq 1, \quad \lambda_t \geq0, \quad \sum_{t=0}^\infty \alpha_t =\infty, \quad \sum_{t=0}^\infty \lambda_t<\infty, \quad \frac{\lambda_t}{\alpha_t} \to 0.
\end{align*}
Then, $v_t\to 0$ a.s., $\EXP{v_t}\to 0$, and for any $\epsilon>0$ and $t>0$
	\[\text{Prob}(v_j \leq \epsilon \hbox{ for all } j\geq t)\geq 1-\frac{1}{\epsilon}\left(\EXP{v_t}+\sum_{i=t}^\infty \lambda_i\right).\]
\end{lemma}
Next, we present self-tuned stepsizes and their properties for the~(\ref{alg:BSSMD}) method.
\begin{pre}
\label{pre:3}
Let $\{\beta_t\}$ be generated by the~(\ref{alg:BSSMD}) method. Let the sets $\mathscr B_i$ be convex and closed such that $\|\beta^i\|\leq M_i$ for all $\beta^i \in \mathscr B_i$ and some $M_i>0$, for all $i$. Let Assumption \ref{ass:3} hold for some $C_i$ large enough such that $C_i^2\mathfrak L_{\omega_i}\geq 8M_i^2\mu_{\omega_i}\mu_F^2$ for all $i$. Let the stepsize $\eta_t$ be given by 
\begin{align*} &\eta_0^*:=\frac{4\mu_Fp_\wedge\sum_{i=1}^lp_i^{-1}\mathfrak L_{\omega_i} M_i^2}{\mathfrak L_{max}\sum_{i=1}^l\mu_{{\omega_i}}^{-1}C_i^2}, \\
&\eta_{t}^*:=\eta_{t-1}^*\left(1-{p_\wedge\mu_F}{\mathfrak L^{-1}_{max}}\eta_{t-1}^*\right), \quad \hbox{for all } t \geq 1.
\end{align*} Then, the following hold:
\begin{itemize}
\item [(a)] The sequence $\{\beta_t\}$ converges a.s. to the unique optimal solution $\beta^*$ of problem~(\ref{eqn:problem1}).
\item [(b)] For any $t \geq 1$, the vector $(\eta_0^*,\ldots,\eta_{t-1}^*)$ minimizes the upper bound of the error $\EXP{\mathcal L(\beta_{t},\beta^*)}$ given in Lemma \ref{lm:bsmd} for all $(\eta_0,\ldots,\eta_{t-1}) \in \left(0,\frac{\mathfrak L_{max}}{2p_\wedge\mu_F}\right]^t$. 

\item [(c)] The~\eqref{alg:BSSMD} method attains the convergence rate $\mathcal O(1/t)$, i.e, for all $t\geq 1$

$\EXP{\|\beta_t-\beta^*\|^2} \leq \frac{p_\vee}{\mu_{min}}\sum_{i=1}^l\frac{C_i^2}{\mu_{\omega_i}}\left(\frac{\mathfrak L_{max}}{p_\wedge\mu_F}\right)^2\frac{1}{t}.$

\item [(d)] Let $\epsilon$ and $\rho$ be arbitrary positive scalars and $T\triangleq 1.5\left(\frac{\mathfrak L_{max}}{p_\wedge\mu_F}\right)^2\sum_{j=1}^l\frac{C_j^2}{\mu_{\omega_j}}\frac{1}{\epsilon\rho}$ we have for all $t\geq T$
\begin{equation*}
\text{Prob}\left(\mathcal L(\beta_{j},\beta^*) \leq \epsilon \hbox{ for all } j\geq t\right)\geq 
1-\rho.
\end{equation*}
\end{itemize}
\end{pre}
\begin{proof} (a)
To show a.s. convergence, we apply Lemma \ref{lm:6}. Consider the inequality \eqref{eq:bsmd_ineq} given by Lemma \ref{lm:bsmd}. Let us define
\begin{align}\label{def:terms1prop3}
&v_t\triangleq \mathcal L (\beta_{t}, \beta^*), \quad \alpha_t\triangleq \frac{2\mu_Fp_\wedge}{\mathfrak L_{max}} \eta_t^*, \quad \lambda_t\triangleq \sum_{i=1}^l\frac{C_{i}^2}{2\mu_{\omega_{i}}}{\eta_t^*}^2.
\end{align}
From definition of $\eta_0^*$ and $C_i^2\mathfrak L_{\omega_i}\geq 8M_i^2\mu_{\omega_i}\mu_F^2$, we have
\begin{align}\label{alphabound}
&\alpha_0=\frac{8\mu_F^2p_\wedge^2\sum_{i=1}^l \frac{\mathfrak L_{\omega_i} M_i^2}{p_i}}{\mathfrak L_{max}^2\sum_{i=1}^l\frac{C_i^2}{\mu_{\omega_i}}}\leq 
\frac{p_\wedge^2\sum_{i=1}^l \frac{\mathfrak L_{\omega_i}^2C_i^2}{p_i\mu_{\omega_i}}}{\mathfrak L_{max}^2\sum_{i=1}^l\frac{C_i^2}{\mu_{\omega_i}}} \leq p_\wedge <1.
\end{align}
Therefore, since $\{\eta_t^*\}$ is non-increasing, we have $0\leq \alpha_t \leq 1$ for all $t\geq 0$. Moreover, from Lemma \ref{lemma:selftunedGeneral}(c), we have that $\sum_{t=0}^\infty \alpha_t=\infty$ and $\sum_{t=0}^\infty \lambda_t<\infty$. Also, the definition of $\alpha_t$ and $\lambda_t$ and that the self-tuned stepsize $\eta_t^*$ has a limit of zero (see proof of Lemma \ref{lemma:selftunedGeneral}, part (c) in the Appendix) imply that $\frac{\lambda_t}{\alpha_t}\to 0$. Therefore, all conditions of Lemma \ref{lm:6} are met and so $\mathcal L(\beta_t, \beta^*) \to 0$ a.s.. The definition of $\mathcal L$ and that $p_i>0$ for all $i$ \fy{imply} that $D_i(\beta_t^i, \beta^{*i}) \to 0$ for all $i$. Using the strong convexity of $\omega_i$ (cf. \eqref{eq:75}), we have $\frac{\mu_{\omega_i}}{2}\|\beta_t^i-\beta^{*i}\|^2\leq D_i(\beta_t^i, \beta^{*i}) $ for all $i$. We conclude that $\beta_t \to \beta^*$ a.s.. \\
\noindent (b) For any $t\geq 1$, let us define the function $e_{t}(\eta_0,\ldots,\eta_{t-1})$ given by the recursion \eqref{eqn:seqer} where $\theta \triangleq\frac{2p_\wedge\mu_F}{\mathfrak L_{max}}$, and $\delta\triangleq\sum_{i=1}^l\frac{C_{i}^2}{2\mu_{\omega_{i}}}$. Also, let $e_0\triangleq 2\sum_{i=1}^lp_i^{-1}\mathfrak L_{\omega_i}M_i^2$. Next, we show that $\mathcal L(\beta_0, \beta^*)\leq e_0$. Using the Lipschitizan property of $\nabla \omega_i$, and the triangle inequality, we have
\begin{align*}
\mathcal L(\beta_0, \beta^*)&=\sum_{i=1}^l\frac{D_i(\beta_0^i,\beta^{*i})}{p_i}\leq \sum_{i=1}^l\frac{\mathfrak L_{\omega_i}}{2p_i}\|\beta_0^i-\beta^{*i}\|_i^2\leq
\sum_{i=1}^l\frac{\mathfrak L_{\omega_i}}{2p_i}\fy{\left(2\|\beta_0^i\|_i^2+2\|\beta^{*i}\|_i^2\right)}\\ 
&\leq\sum_{i=1}^l\frac{2\mathfrak L_{\omega_i}M_i^2}{p_i}=e_0.
\end{align*}
From $\mathcal L(\beta_0, \beta^*)\leq e_0$, relations \eqref{eqn:seqer}, \eqref{eq:bsmd_ineq} and using induction, it can be seen that $\EXP{\mathcal L(\beta_t, \beta^*) }\leq e_{t}(\eta_0,\ldots,\eta_{t-1})$ for all $t\geq 0$ and any arbitrary $(\eta_0,\ldots,\eta_{t-1}) \in \left(0,\frac{\mathfrak L_{max}}{2p_\wedge\mu_F}\right]^t$.  Therefore, $e_{t}$ is a well-defined upper bound for the algorithm. To complete the proof, it suffices to show that the conditions of Lemma \ref{lemma:selftunedGeneral} hold. First we show that $e_0 \leq \frac{2\delta}{\theta^2}$. From the values of $e_0$, $\eta_0^*$, $\theta$, and $\delta$, we have $\eta_0^*=\frac{\theta}{2\delta}e_0$. From \ap{the} definition of $\alpha_0$ in \eqref{def:terms1prop3} and \eqref{alphabound}, we have $\alpha_0=\theta \eta_0^* <1$. By two preceding relations we obtain $e_0 \leq \frac{2\delta}{\theta^2}$. Hence, conditions of Lemma \ref{lemma:selftunedGeneral} hold. From Lemma \ref{lemma:selftunedGeneral}(a), we conclude the desired result.\\
\noindent (c) Following the proof of part (b), from Lemma \ref{lemma:selftunedGeneral}(b) and definitions of $\delta$ and $\theta$ in part (b), we obtain for all $t\geq 1$
\begin{align}\label{ineq:boundOnL}
\EXP{\mathcal L(\beta_t,\beta^*)} &\leq e_t\leq
\left(\frac{\mathfrak L_{max}}{p_\wedge\mu_F}\right)^2\sum_{i=1}^l\frac{C_i^2}{2\mu_{\omega_i}}\frac{1}{t}.
\end{align}
Note that from strong convexity of $\omega_i$ we have
\begin{align*}
&\mathcal L(\beta_t, \beta^*)=\sum\nolimits_{i=1}^lp_i^{-1}D_i(\beta_t^i,\beta^{*i})\geq \sum\nolimits_{i=1}^l p_i^{-1}0.5{\mu_{\omega_i}}\|\beta_t^i-\beta^{*i}\|_i^2\geq \mu_{min}({2p_\vee})^{-1}\|\beta_t-\beta^*\|^2.
\end{align*}
Combining the two preceding relations completes the proof.\\
\noindent (d) We use the probabilistic bound given in Lemma \ref{lm:6}. First we estimate the term $\sum_{i=t}^\infty \lambda_i$ where $\lambda_i$ is given by \eqref{def:terms1prop3}. Note that Lemma \ref{lemma:selftunedGeneral}(b) implies $\eta_i^*\leq\frac{2}{\theta i}$. Therefore, we can write 
\begin{align}
\label{eq:integral}
\sum_{i=t}^\infty \lambda_i=\sum_{i=t}^\infty\sum_{j=1}^l\frac{C_j^2}{2\mu_{\omega_j}}{\eta_i^*}^2 &\leq \sum_{j=1}^l\frac{C_j^2}{2\mu_{\omega_j}}\sum_{i=t}^\infty \left(\frac{\mathfrak L_{max}}{p_\wedge\mu_F\ i}\right)^2
 \leq \left(\frac{\mathfrak L_{max}}{p_\wedge\mu_F}\right)^2\sum_{j=1}^l\frac{C_j^2}{2\mu_{\omega_j}}\left(\frac{1}{t}+\int_{t}^\infty\frac{1}{x^2}dx\right)\notag\\
&=\left({\mathfrak L_{max}}({p_\wedge\mu_F})^{-1}\right)^2\sum_{j=1}^l{C_j^2}{\mu_{\omega_j}}^{-1}\left({1}/{t}\right).
\end{align}
By \eqref{eq:integral}, \eqref{ineq:boundOnL}, and Lemma \ref{lm:6}, we obtain the desired relation.
\end{proof}
Under a uniform distribution, i.e., $p_i=\frac{1}{l}$ for $i=1,\ldots,l$, Proposition \ref{pre:3} indicates that $\EXP{\|\beta_t-\beta^*\|^2}\to 0$ with the order of $\mathcal{O}\left(\frac{l}{t}\right)$.This is similar to the error bound derived in \cite{dang2015stochastic} for stochastic block mirror descent (SBMD) method (cf. Corollary 2.5 in \cite{dang2015stochastic}). Next, we compare the constant factor of the error bound derived in \cite{dang2015stochastic} with that of~\eqref{alg:BSSMD} method.  
\begin{com}
Let Assumption \ref{ass:3} hold for some unknown $C_i>0$ for all $i$. Let $\beta_t$ be generated by algorithm~\eqref{alg:BSSMD} where $\mathfrak L_{\omega_i}=\mathfrak L_{\omega}$ and $\mu_{\omega_i}=\mu_{\omega}$ for all $1\leq i\leq l$ and $\bar\beta_t$ be generated by SBMD method in \cite{dang2015stochastic}. Then, By Lemma \ref{lm:2}, we have $\EXP{\|\bar\beta_t-\beta^*\|^2} \leq \frac{2}{\mu_F\mu_\omega}\EXP {F(\bar\beta_t)-F(\beta^*)}$ and by Corollary 2.5 in \cite{dang2015stochastic}, we have $\EXP {F(\bar\beta_t)-F(\beta^*)} \leq \frac{2l \mathfrak L_{\omega}}{\mu_F}\sum_{i=1}^l C_i^2 \left(\frac{1}{t+1}\right)$. Combining the preceding inequalities, we obtain for all $t \geq 1$
\begin{align} 
\label{eq:lan}
\EXP{\|\bar\beta_t-\beta^*\|^2} \leq \frac{4l \mathfrak L_{\omega}}{\mu_F^2\mu_\omega}\sum_{i=1}^l C_i^2 \left(\frac{1}{t+1}\right).
\end{align}
On the other hand, by Proposition \ref{pre:3}, we have for all $t \geq 1$
\begin{align} 
\label{eq:ourbound}
\EXP{\|\beta_t-\beta^*\|^2} &\leq \frac{l\mathfrak L_{\omega}^2}{\mu_{\omega}^2\mu_F^2}\sum_{i=1}^l C_i^2 \left(\frac{1}{t+1}\right).
\end{align}
Comparing \eqref{eq:lan} and \eqref{eq:ourbound}, \ap{we} note that the constant factor of the error bound of \eqref{alg:BSSMD} method is smaller when $\frac{\mathfrak L_{\omega}}{\mu_\omega}<4$. In particular, for SGD method where $\mathfrak L_{\omega}=\mu_\omega=1$, it can be \fy{four} times better than the constant factor \fy{of} SBMD in \cite{dang2015stochastic}.
\end{com}
{
\begin{com} Proposition \ref{pre:3} states that the self-tuned stepsizes not only guarantee the convergence of the \eqref{alg:BSSMD} method, but also the constant factor provided in part (c) is the minimum for any arbitrary stepsize rule within a given range. \ap{For} example, compare this constant factor with that of the stochastic subgradient method under harmonic stepsize rules in~\cite{nemirovski2009robust}. In that work (see relations (2.9)-(2.10)), under the harmonic update rule for stepsizes given by $\eta_t=\gamma/t$ for some constant $\gamma>1/(2\mu_F)$, it is shown that 
\begin{equation}
\label{eq:58}
\EXP{\| \beta_t-\beta^*\|_2^2} \leq \max\left\{\frac{\gamma^2 C^2}{2\mu_F\gamma-1},\| \beta_0-\beta^*\|_2^2\right\}\frac{1}{t}.
\end{equation}
Let $l=1$. Here we show that for any arbitrary $\gamma>\frac{1}{2\mu_F}$, the term $\frac{\gamma^2 C^2}{2\mu_F\gamma-1}$ is larger than the constant factor of the self-tuned stepsizes that is $\left(\frac{C \mathfrak L_\omega}{\mu_\omega\mu_F}\right)^2$. Note \ap{that} in the stochastic subgradient method $\omega(\beta):=\frac{\|\beta\|_2^2}{2}$. This implies $\mu_\omega=\mathfrak L_\omega=1$. Then,
\[\frac{\text{ Harmonic constant factor}}{\text{Self-tuned constant factor}}=\frac{\gamma^2 C^2\mu_F^2}{(2\mu_F\gamma-1)C^2}=\frac{\gamma^2\mu_F^2}{2\mu_F\gamma-1}.\] 
Note that $\gamma^2\mu_F^2-2\mu_F\gamma+1=(\gamma \mu_F-1)^2>0$ for all $\gamma>\frac{1}{2\mu_F}$. Therefore, the preceding relation implies that the harmonic constant factor in~\cite{nemirovski2009robust} is larger than the self-tuned constant factor for any arbitrary $\gamma>\frac{1}{2\mu_F}$.
\end{com}
}
%\label{sec_RSS}
\subsection{Self-tuned randomized block gradient SMD method}\label{subsec:diff}
In this section, we assume the objective function in problem~\eqref{eqn:problem1} is differentiable and has Lipschitz gradients. Our goal is to utilize this property and develop a self-tuned scheme that is characterized with the problem parameters and algorithm settings. {To solve problem~\eqref{eqn:problem1}, we consider the randomized block gradient SMD method as follows
\begin{equation}\label{alg:BSGMD}\tag{RB-GSMD}
					\beta_{t+1}^{i}=\left\{
					\begin{array}{ll}
           \mathcal P_{i_t}\big(\beta_{t}^{i_t}, \eta_t g_{i_t}(\beta_t)\big) & \text{if } i = i_t,\\
           \beta_{t}^{i}& \text{if } i \neq i_t,
     \end{array} \right.
				 \end{equation} 
for all $t \geq 0$, where $g_{i_t}(\beta_t)$ is the $i_t$th block of the gradient of the stochastic function $f(\cdot,\xi_t)$ at $\beta_t$.
Throughout this section, we let $F$ have Lipschitz gradients with parameter  $\mathfrak L_F>0$. We also define the stochastic errors $z_{t}$ as follows
\begin{equation}
\label{eq:z}
z_{t}\triangleq \fy{g(\beta_{t})-\nabla F(\beta_{t})}.
\end{equation}
%\begin{equation}
%\label{eq:200}
%\beta^{*}=\underset{\beta\in \mathscr B}{\text {argmin}}\{{\eta_t}\langle \mathbf g_*, \beta-\beta^*\rangle+D_{\omega}(\beta^*,\beta)\}.
%\end{equation}
Next, we state the main assumptions on stochastic gradients.
\begin{assumption}
\label{ass:33}
The errors $z_t$ are such that a.s. we have \nm{$\EXP{z_t\mid \mathcal{F}_t}=0$} for all $t \geq 0$. Moreover, there exists some $\nu_i>0$ for all $i$ such that $ \fy{\EXP{\| z_t^{i}\|_{*i}^2|\mathcal F_t}}\leq  \nu_i^2,$ for all $t \geq 0$.
\end{assumption}
}
Next, we have the lemma that provides a recursive bound on the error of the algorithm.
\begin{lemma}\label{lem:ineqSmooth}
Let Assumption \ref{ass:33} hold and $\beta_t$ be generated by the~\eqref{alg:BSGMD} method. We have a.s. for all $t \geq 0$
\begin{align}
\label{eq:ineqSmooth}
&\EXP{\mathcal L (\beta_{t+1}, \beta^*)|\mathcal F_t} \leq(1-\eta_t{2\mu_Fp_\wedge}{\mathfrak L^{-1}_{max}}+\eta_t^2{2{\mathfrak L_F}^2p_\vee}{\mu^{-2}_{min}})\mathcal L (\beta_{t}, \beta^*)+\eta_t^2\sum\nolimits_{i=1}^l{\nu_{i}^2}{\mu^{-1}_{\omega_{i}}}.
\end{align}
\end{lemma}
\begin{proof}
Consider the update rule \eqref{alg:BSGMD}. Writing the first-order optimality \fy{condition}, we have for all $\beta \in \mathscr B$
\begin{equation}
\label{eq:202}
\langle{\eta_t} g_{i_t} +\nabla D_{i_t}(\beta_t^{i_t}, \beta_{t+1}^{i_t}), \beta^{i_t}-\beta_{t+1}^{i_t}\rangle \geq 0, 
\end{equation}
Using equation \eqref{eq:73}, from \eqref{eq:202} we obtain for all $\beta \in \mathscr B$
\begin{align}
\label{eq:76}
&\langle\nabla\fy{\omega_{i_t}}(\beta_{t+1}^{i_t})-\nabla\fy{\omega_{i_t}}(\beta_{t}^{i_t}), \beta^{i_t}-\beta_{t+1}^{i_t}\rangle\geq {\eta_t}\langle g_{i_t} ,\beta_{t+1}^{i_t}-\beta^{i_t}\rangle.
\end{align}
Let $\beta:=\beta^*$ in relation \eqref{eq:76}. Adding and subtracting the term $\eta_t\langle \nabla F_{i_t}(\beta^{*}) , \beta_{t+1}^{i_t} -\beta^{*{i_t}}\rangle$, we get
\begin{align}
\label{eq:204}
\langle\nabla\fy{\omega_{i_t}}(\beta_{t+1}^{i_t})-\nabla\fy{\omega_{i_t}}(\beta_{t}^{i_t}), \beta^{*i_t}-\beta_{t+1}^{i_t}\rangle  &\geq {\eta_t}\langle g_{i_t}-\nabla F({\beta^*}^{i_t}),\beta_{t+1}^{i_t}-\beta^{*i_t}\rangle\notag \\
&+\eta_t\langle \nabla F_{i_t}(\beta^{*}) , \beta_{t+1}^{i_t} -\beta^{*{i_t}}\rangle.
\end{align}
From relation \eqref{eq:74}, we get
\begin{align*}
&\langle\nabla\fy{\omega_{i_t}}(\beta_{t+1}^{i_t})-\nabla\fy{\omega_{i_t}}(\beta_{t}^{i_t}),\beta^{*{i_t}}-\beta_{t+1}^{i_t}\rangle=D_{i_t}(\beta_t^{i_t}, \beta^{*i_t})-D_{i_t}(\beta_{t+1}^{i_t}, \beta^{*i_t})-D_{i_t}(\beta_t^{i_t}, \beta_{t+1}^{i_t}).
\end{align*}
Therefore, from the preceding relation, \eqref{eq:204}, and relation \eqref{eq:75},
%\begin{equation*}
%D_\omega(\beta_t, \beta^*)-D_\omega(\beta_{t+1}, \beta^*)-D_\omega(\beta_t, \beta_{t+1}) \geq {\eta_t}\langle g_{i_t}-\nabla f(\beta^*)  ,\beta_{t+1}-\beta^*\rangle.
%\end{equation*}
%By the strong convexity of $\omega(\beta)$ and relation~(\ref{eq:75}), we get
\begin{align}
\label{eq:205}
&D_{i_t}(\beta_t^{i_t}, \beta^{*i_t})-D_{i_t}(\beta_{t+1}^{i_t}, \beta^{*i_t})-\frac{\mu_{\omega_{i_t}}}{2}\fy{\|\beta_t^{i_t}-\beta_{t+1}^{i_t}\|_{i_t}^2}-\eta_t\langle\nabla F_{i_t}(\beta^{*}) , \beta_{t+1}^{i_t} -\beta^{*{i_t}}\rangle\geq\notag\\ 
& {\eta_t}\langle g_{i_t}-\nabla F_{i_t}(\beta^{*})  ,\beta_{t+1}^{i_t}-\beta^{*i_t}\rangle
\end{align}
Next, we find a lower bound for the right-hand side term. By adding and subtracting $\eta_t\langle g_{i_t}-\nabla F_{i_t}(\beta^{*}), \beta_{t}^{i_t}\rangle$, we get
\begin{align}
\label{eq:206}
&{\eta_t}\langle g_{i_t}-\nabla F_{i_t}(\beta^{*}) ,\beta_{t+1}^{i_t}-\beta_t^{i_t}\rangle
%& \geq -\Big|\langle \frac{\eta_t}{\sqrt{\mu_{\omega_{i_t}}}}(\tilde g_t- \mathbf g_*) ,\sqrt{\mu_{\omega_{i_t}}}(\beta_{t+1}-\beta_t)\rangle\Big|+{\eta_t}\langle \tilde g_t- \mathbf g_* ,\beta_{t}-\beta^*\rangle \nonumber\\
+{\eta_t}\langle g_{i_t}-\nabla F_{i_t}(\beta^{*}),\beta_{t}^{i_t}-\beta^{*i_t}\rangle \geq  \frac{-\eta_t^2}{2\mu_{\omega_{i_t}}}\fy{\| g_{i_t}-\nabla F_{i_t}(\beta^{*}) \|^2_{*i_t}}\nonumber\\&-\frac{\mu_{\omega_{i_t}}}{2}\fy{\|\beta_{t+1}^{i_t}-\beta_{t}^{i_t}\|_{i_t}^2}+{\eta_t}\langle g_{i_t}-\nabla F_{i_t}(\beta^{*}) ,\beta_{t}^{i_t}-\beta^{*i_t}\rangle,
\end{align}
where the last inequality follows from Fenchel's inequality, i.e., $|\langle x,y\rangle|\leq \frac{1}{2}\| x\|^2+\frac{1}{2}\| y\|^2_*$.
Combining \eqref{eq:205} and \eqref{eq:206} yields
\begin{align*}
D_{i_t}(\beta_{t+1}^{i_t}, \beta^{*i_t}) &\leq D_{i_t}(\beta_t^{i_t}, \beta^{*i_t})-{\eta_t}(\langle g_{i_t}-\nabla F_{i_t}(\beta^{*}) ,\beta_{t}^{i_t}-\beta^{*i_t}\rangle+\langle \nabla F_{i_t}(\beta^{*}) , \beta_{t+1}^{i_t} -\beta^{*{i_t}}\rangle)\\&+\frac{\eta_t^2\fy{\| g_{i_t}-\nabla F_{i_t}(\beta^{*}) \|^2_{*i_t}}}{2\mu_{\omega_{i_t}}}.
\end{align*}
Using relation~(\ref{eq:z}), and invoking the triangle inequality and relation $(a+b)^2\leq2a^2+2b^2$ for any $a,b \in \mathbb R$, we obtain
\begin{align*}
D_{i_t}(\beta_{t+1}^{i_t}, \beta^{*i_t}) &\leq D_{i_t}(\beta_t^{i_t}, \beta^{*i_t})-\eta_t\langle \nabla F_{i_t}(\beta^{*}),\beta_{t+1}^{i_t}-\beta^{*{i_t}}\rangle\\
&-{\eta_t}\langle \nabla F_{i_t}(\beta_t)-\nabla F_{i_t}(\beta^{*}) \fy{+z_t^{i_t}},\beta_{t}^{i_t}-\beta^{*i_t}\rangle
+{\eta_t^2}{\mu^{-1}_{\omega_{i_t}}}\fy{\| \nabla F_{i_t}(\beta_t)-\nabla F_{i_t}(\beta^{*}) \|_{*i_t}^2}\\&+{\eta_t^2}{\mu^{-1}_{\omega_{i_t}}}\fy{\|z_t^{i_t}\|^2_{*i_t}}.
\end{align*}
From the preceding relation, the definition of the function $\mathcal L$, and that $\beta_{t+1}^{i}=\beta_t^i$ for all $i\neq i_t$, we have 
\begin{align*}
\mathcal L(\beta_{t+1}, \beta^*)&=\sum_{i\neq i_t} p_i^{-1}D_i(\beta_{t+1}^i, \beta^{*i})+ p_{i_t}^{-1}D_{i_t}(\beta_{t+1}^{i_t}, \beta^{*i_t}) \\
&\leq\mathcal L (\beta_{t}, \beta^*)+p_{i_t}^{-1}\bigg(-{\eta_t}\langle \nabla F_{i_t}(\beta_t)-\nabla F_{i_t}(\beta^{*}) +\fy{z_t^{i_t}},\beta_{t}^{i_t}-\beta^{*i_t}\rangle\\&-\eta_t\langle \nabla F_{i_t}(\beta^{*}) , \beta_{t+1}^{i_t} -\beta^{*{i_t}}\rangle+{\eta_t^2}{\mu^{-1}_{\omega_{i_t}}}\| \nabla F_{i_t}(\beta_t)-\nabla F_{i_t}(\beta^{*}) \|^2_{*i_t}+{\eta_t^2}{\mu^{-1}_{\omega_{i_t}}}\|\fy{z_t^{i_t}}\|^2_{*i_t}\bigg).
\end{align*}
%Under Lipschitzian property of $\nabla F$ with parameter $\mathfrak L_F$, and strong convexity of $F$ with parameter $\mu_F$, we get
%\begin{align*}
%&\mathcal L(\beta_{t+1}, \beta^*)\leq\mathcal L (\beta_{t}, \beta^*)+p_{i_t}^{-1}\bigg(-{\eta_t\mu_F}\|\beta_{t}^{i_t}-\beta^{*i_t}\|^2\\
%&+\frac{\eta_t^2\mathfrak L_F^2}{\mu_{\omega_{i_t}}}\|\beta_{t}^{i_t}-\beta^{^{*i_t}}\|^2
%-{\eta_t} \langle z_t,\beta_{t}^{i_t}-\beta^{*i_t}\rangle+\frac{\eta_t^2}{\mu_{\omega_{i_t}}}\|z_t\|^2_*\bigg).
%\end{align*}
Taking conditional expectations from both sides of the preceding relation on $\mathcal F_t\cup \{i_t\}$, we get 
\begin{align*}
\EXP{\mathcal L(\beta_{t+1}, \beta^*)\mid \mathcal F_t\cup \{i_t\}} &\leq \mathcal L (\beta_{t}, \beta^*)+p_{i_t}^{-1}\eta_t(-\langle \nabla F_{i_t}(\beta^{*}),\beta_{t+1}^{i_t} -\beta^{*{i_t}}\rangle \\&+\left\langle \EXP{\fy{z_t^{i_t}}\mid \mathcal F_t\cup \{i_t\}},\beta^{*i_t}-\beta_t^{i_t}\right\rangle)\\
&+p_{i_t}^{-1}{\eta_t}\langle \nabla F_{i_t}(\beta_t)-\nabla F_{i_t}(\beta^{*}),\beta^{*i_t}-\beta_{t}^{i_t}\rangle\\&+p_{i_t}^{-1}{\eta_t^2}{\mu^{-1}_{\omega_{i_t}}}\left(\| \nabla F_{i_t}(\beta_t)-\nabla F_{i_t}(\beta^{*}) \|^2_{*i_t}+\EXP{\|\fy{z_t^{i_t}}  \|^2_{*i_t}\mid \mathcal F_t\cup \{i_t\}}\right).
\end{align*} 
Assumption \ref{ass:33} implies that $\EXP{\fy{z_t^{i_t}}\mid \mathcal F_t}=0$. Using that and the bound provided in Assumption \ref{ass:33}, we obtain 
\begin{align*}
\EXP{\mathcal L (\beta_{t+1}, \beta^*)\mid \mathcal F_t\cup \{i_t\}}&\leq \mathcal L (\beta_{t}, \beta^*)-p_{i_t}^{-1}\eta_t(\langle \nabla F_{i_t}(\beta^{*}),\beta_{t+1}^{i_t} -\beta^{*{i_t}}\rangle\\&+\langle \nabla F_{i_t}(\beta_t)-\nabla F_{i_t}(\beta^{*}),\beta^{*i_t}-\beta_{t}^{i_t}\rangle)\\
&+p_{i_t}^{-1} {\eta_t^2}{\mu^{-1}_{\omega_{i_t}}}\left(\| \nabla F_{i_t}(\beta_t)-\nabla F_{i_t}(\beta^{*}) \|^2_{*i_t}
+{\nu_{i_t}^2}\right).
\end{align*} 
Next, taking expectations with respect to $i_t$, we obtain
\begin{align*}
\EXP{\mathcal L (\beta_{t+1}, \beta^*)\mid \mathcal F_t} &\leq \mathcal L (\beta_{t}, \beta^*)+\eta_t
\left(\left\langle \nabla F(\beta_t)-\nabla F(\beta^{*}) ,\beta^*-\beta_t\right\rangle-\langle \nabla F(\beta^{*}) , \beta_{t+1}-\beta^{*}\rangle\right) \notag\\
&+{\eta_t^2}{\mu^{-1}_{min}}\| \nabla F(\beta_t)-\nabla F(\beta^{*}) \|^2_*
+\eta_t^2\sum\nolimits_{i=1}^l {\nu_{i}^2}{\mu^{-1}_{\omega_i}},
\end{align*} 
where we use the definition of $\left\langle \cdot, \cdot\right\rangle$ given in the notation. 
Using \fy{the optimality condition} for problem~\eqref{eqn:problem1} and under \ap{the} Lipschitzian property of $\nabla F$ and strong convexity of $F$,
\begin{align}
\label{ineq:lemma6ineq1}
\EXP{\mathcal L (\beta_{t+1}, \beta^*)\mid \mathcal F_t} &\leq \mathcal L (\beta_{t}, \beta^*)-{\eta_t\mu_F}\|\beta_{t}-\beta^*\|^2+{\eta_t^2\mathfrak L_F^2}{\mu_{min}}^{-1}\|\beta_{t}-\beta^*\|^2\notag\\&+\eta_t^2\sum\nolimits_{i=1}^l{\nu_{i}^2}{\mu_{\omega_i}}^{-1}.
\end{align}
From the definition of norm $\|\cdot\|$, we can write
\begin{align*}
\|\beta_{t}-\beta^*\|^2=\sum\nolimits_{i=1}^l\|\beta_t^i-\beta^{*i}\|_i^2\geq 2\sum\nolimits_{i=1}^l {D_i(\beta_t^i,\beta^{*i})}{\mathfrak L^{-1}_{\omega_i}}&\geq {2p_\wedge}{\mathfrak L^{-1}_{max}}\sum\nolimits_{i=1}^l p_i^{-1}D_i(\beta_t^i,\beta^{*i})\notag\\
&={2p_\wedge}{\mathfrak L^{-1}_{max}}\mathcal{L}(\beta_t,\beta^{*}),
\end{align*} 
where in the first inequality we used relation \eqref{eq:75}, and in the last relation we used the definition of function $\mathcal L$. 
Similarly,
\begin{align*}
%\label{eq:akhar}
&\|\beta_{t}-\beta^*\|^2\leq 2\sum_{i=1}^l\frac{D_i(\beta_t^i,\beta^{*i})}{\mu_{\omega_i}}\leq\frac{2p_\vee}{\mu_{min}}\mathcal{L}(\beta_t,\beta^{*}),
\end{align*} 
From \fy{the last three relations}, we obtain the desired inequality.
%Note that since $\eta_t\leq \frac{\mu_F \mu_\omega^2}{2\mathfrak L_f^2\mathfrak L_\omega }$, we have $\frac{2\eta_t^2\mathfrak L_f^2}{\mu_\omega^2}\leq \frac{\eta_t\mu_F}{\mathfrak L_\omega}$. Therefore, we obtain for all $t\geq 0$
%\begin{eqnarray*}
%\EXP{D_\omega(\beta_{t+1}, \beta^*)\mid \mathcal F_t} \leq \left(1-\frac{\eta_t\mu_F}{\mathfrak L_\omega}\right)D_\omega(\beta_t, \beta^*)+\frac{\nu^2\eta_t^2}{\mu_\omega}.
%\end{eqnarray*}
\end{proof}
\begin{comment}
Inequality \eqref{eq:ineqSmooth} provides a closed-form function for an upper bound of the error of the \eqref{alg:BSGMD} scheme. Comparing this relation with the result of Lemma \ref{lm:bsmd}, we observe that the inequalities differ from two aspects: (i) the \textit{contraction term} multiplied by the term $\mathcal L_\omega(\beta_t,\beta^*)$ in the nonsmooth case is smaller than that in the smooth case; (ii) the upper bound in the smooth case is independent of the bound on the gradient, i.e., constant $C_i$. Instead the relation is characterized by the bound on the stochastic errors, that is denoted by $\nu_i$. 
\end{comment}
Next, we present self-tuned stepsizes for the \eqref{alg:BSGMD} method and show their properties.
\begin{pre}
\label{pre:2}
Let $\{\beta_t\}$ be generated by the~\eqref{alg:BSGMD} method. Let the set $\mathscr B_i$ be convex and closed such that $\|\beta^i\|\leq M_i$ for all $\beta_i \in \mathscr B_i$ and some $M_i>0$. Let Assumption \ref{ass:33} hold for some $\nu_i>0$, and the stepsize $\eta_t$ be given by 
\begin{align*} 
&\eta_0^*:=\frac{4\mu_Fp_\wedge\sum_{i=1}^lp_i^{-1}\mathfrak L_{\omega_i} M_i^2}{\mathfrak L_{max} \left(\frac{8{\mathfrak L_F}^2p_\vee}{{\mu_{min}}^2}\sum_{i=1}^l p_i^{-1}\mathfrak L_{\omega_i} M_i^2+\sum_{i=1}^l\frac{2\nu_{i}^2}{\mu_{\omega_{i}}}\right)},
\end{align*} 
\begin{align*}
&\eta_{t}^*:=\eta_{t-1}^*\left(1-{p_\wedge\mu_F}{\mathfrak L^{-1}_{max}}\eta_{t-1}^*\right), \quad \hbox{for all } t \geq 1.
\end{align*}
Then, the following hold:
\begin{itemize}
\item [(a)] The sequence $\{\beta_t\}$ generated by the \eqref{alg:BSGMD} method converges a.s. to the unique optimal solution $\beta^*$ of problem~(\ref{eqn:problem1}).
\item [(b)] For any $t \geq 1$, the vector $(\eta_0^*, \eta_1^*,\ldots,\eta_{t-1}^*)$ minimizes the upper bound of the error $\EXP{D_\omega(\beta_{t},\beta^*)}$ given in Lemma \ref{lem:ineqSmooth} for all $(\eta_0, \eta_1,\ldots,\eta_{t-1}) \in \left(0,\frac{\mathfrak L_{max}}{2\mu_F p_\wedge}\right]^t$. 

\item [(c)] The \eqref{alg:BSGMD} method attains the convergence rate $\mathcal O(1/t)$, i.e, for all $t\geq 1$
\begin{align*}
&\EXP{\|\beta_t-\beta^*\|^2} \leq 2\left(\frac{ \mathfrak L_{max}}{p_\wedge\mu_F}\right)^2\left(4\frac{{\mathfrak L_F}^2p_\vee}{{\mu_{min}}^2}\sum_{i=1}^l p_i^{-1}\mathfrak L_{\omega_i} M_i^2+\sum_{i=1}^l\frac{\nu_{i}^2}{\mu_{\omega_{i}}}\right)\frac{1}{t}.
 \end{align*}

\item [(d)] Let $\epsilon$ and $\rho$ be arbitrary positive scalars and $T\triangleq \frac{1.5}{\epsilon\rho}\left(\frac{\mathfrak L_{max}}{p_\wedge\mu_F}\right)^2\times$\\$\left(\frac{{8\mathfrak L_F}^2p_\vee}{{\mu_{min}}^2}\sum_{j=1}^l p_j^{-1}\mathfrak L_{\omega_j} M_j^2+\sum_{j=1}^l\frac{2\nu_{j}^2}{\mu_{\omega_{j}}}\right)$ we have for all $t\geq T$
\begin{equation*}
\text{Prob}\left(\mathcal L(\beta_{j},\beta^*) \leq \epsilon \hbox{ for all } j\geq t\right)\geq 
1-\rho.
\end{equation*}
\end{itemize}
\end{pre}
\begin{proof}
Consider relation \eqref{eq:ineqSmooth}. Taking expectations from both sides and rearranging the terms, we can write 
\begin{align*}
\EXP{\mathcal L (\beta_{t+1}, \beta^*)}&\leq \left(1-\eta_t{2\mu_Fp_\wedge}{\mathfrak L^{-1}_{max}}\right)\EXP{\mathcal L (\beta_{t}, \beta^*)}+2\eta_t^2{{\mathfrak L_F}^2p_\vee}{{\mu^{-2}_{min}}}\EXP{\mathcal L (\beta_{t}, \beta^*)}\\
&+\eta_t^2\sum\nolimits_{i=1}^l{\nu_{i}^2}{\mu^{-1}_{\omega_{i}}}.
\end{align*}
From relation \eqref{eq:75}, and \fy{the} triangle inequality, we have 
\begin{align*}
\mathcal L(\beta_t, \beta^*)&  \leq \sum_{i=1}^l p_i^{-1} \frac{\mathfrak L_{\omega_i}}{2}\fy{\|\beta_t^i-\beta^{*i}\|_{i}^2} \leq 2\sum_{i=1}^l p_i^{-1}\mathfrak L_{\omega_i} M_i^2.
\end{align*}
From the preceding inequalities, we obtain
\begin{align*}
&\EXP{\mathcal L (\beta_{t+1}, \beta^*)}\leq \left(1-\eta_t{2\mu_Fp_\wedge}{\mathfrak L^{-1}_{max}}\right)\EXP{\mathcal L (\beta_{t}, \beta^*)}+\left(\frac{8{\mathfrak L_F}^2p_\vee}{{\mu_{min}}^2}\sum_{i=1}^l p_i^{-1}\mathfrak L_{\omega_i} M_i^2+\sum_{i=1}^l\frac{2\nu_{i}^2}{\mu_{\omega_{i}}}\right)\frac{1}{2}\eta_t^2.
\end{align*}
Let us define $C^2\triangleq{\sum_{i=1}^l {\mu_{\omega_i}}^{-1} C_i^2}$ and $\bar C^2$ such that $\bar C^2\triangleq \frac{8{\mathfrak L_F}^2p_\vee}{{\mu_{min}}^2}\sum_{i=1}^l p_i^{-1}\mathfrak L_{\omega_i} M_i^2+\sum_{i=1}^l\frac{2\nu_{i}^2}{\mu_{\omega_{i}}}$. Note that the preceding inequality is similar to the relation \eqref{eq:bsmd_ineq}, where $C^2$ is replaced by the term $\bar C^2$. Therefore, the desired results here follow by only substituting $C^2$ by $\bar C^2$ in Proposition \ref{pre:3}. \ap{It only remains to show} that: (i) $\eta_0^*=\frac{4\mu_F p_\wedge\sum_{i=1}^lp_i^{-1}\mathfrak L_{\omega_i} M_i^2}{\mathfrak L_{max} \bar C^2}$, and (ii) the conditions of Proposition \ref{pre:3} also hold for $\alpha_0$. The relation (i) holds directly from definition of $\eta_0^*$ given by Proposition \ref{pre:2} and the definition of $\bar C$. To show (ii), we need to verify that $\alpha_0<1$.  From definition of $\alpha_0$ \fy{given by \eqref{def:terms1prop3}}, we have
\begin{align*}
\alpha_0&={2\mu_F p_\wedge}{\mathfrak L^{-1}_{max}}\eta_0^*={2\mu_F p_\wedge}{\mathfrak L^{-1}_{max}}\times\frac{4\mu_Fp_\wedge\sum_{i=1}^lp_i^{-1}\mathfrak L_{\omega_i} M_i^2}{\mathfrak L_{max}\left(\frac{8{\mathfrak L_F}^2p_\vee}{{\mu_{min}}^2}\sum_{i=1}^l p_i^{-1}\mathfrak L_{\omega_i} M_i^2+\sum_{i=1}^l\frac{2\nu_{i}^2}{\mu_{\omega_{i}}}\right)}
\\&
=\frac{\mu_F^2 p^2_{\wedge}}{\mathfrak L^2_{max}}\times\frac{\sum_{i=1}^l p_i^{-1}\mathfrak L_{\omega_i} M_i^2}{\left(\frac{{\mathfrak L_F}^2p_\vee}{{\mu_{min}}^2}\sum_{i=1}^l p_i^{-1}\mathfrak L_{\omega_i} M_i^2\right)}=\frac{\mu^2_F }{\mathfrak L^2_F}\frac{p^2_{\wedge}}{p_\vee} \frac{\mu^2_{min}}{\mathfrak L^2_{max}}<1\ap{,}
\end{align*}
where the last relation follows since $\mu_F \leq \mathfrak L_F$ and $\mu_{\omega_i} \leq \mathfrak L_{\omega_i}$. Therefore, the conditions of Proposition \ref{pre:3} hold for $\alpha_0$ and the desired results follow.
\end{proof}
\subsection{Unifying self-tuned stepsizes}
\label{subsec:unifying}
Recall that Propositions \ref{pre:3} and \ref{pre:2} provide self-tuned stepsize rules for the case where problem \eqref{eqn:problem1} is nonsmooth and smooth, respectively. These update rules are characterized in terms of problem parameters such as $C_i,\nu_i,\mu_F$ and algorithm settings such as $\mu_{\omega_i}, \mathfrak L_{\omega_i}$. An important question is how we may employ such self-tuned stepsize rules \fy{when} some of the problem parameters are not known in advance, or are challenging to estimate? Here, our goal is to develop a unifying class of self-tuned stepsize rules that can be employed in both smooth and nonsmooth cases when some of the problem parameters are unavailable. Assume that $l=1$. Let us compare the stepsize rules in Propositions \ref{pre:3} and \ref{pre:2}. We observe that although the initial stepsize $\eta_0^*$ is different, both schemes share the same tuning rule. We also observe that the only problem parameter that is needed to be known for the tuning update rule is $\mu_F$. This parameter is known in advance in many applications such as SVM. Note that $\mathfrak L_\omega$ is the Lipschitzian parameter associated with the prox mapping and depends on the choice of the distance generating function $\omega$. This function is user-specified. 
%For example, for stochastic subgradient/gradient methods we set $\omega(\beta):=\frac{1}{2}\|\beta\|_2^2$, and therefore $\mathfrak L_\omega=1$. 
In practice, when problem parameters such as $M,C$, or $\mathfrak L_F$ are unavailable or difficult to estimate, the initial stepsize $\eta_0^*$ cannot be evaluated. In such cases one may choose $\eta_0^*$ arbitrarily and still use the update rule $\eta_{t}^*:=\eta_{t-1}^*\left(1-(\mu_F/\mathfrak L_\omega)\eta_{t-1}^* \right)$. We show that even under this relaxation, some of the main properties of the self-tuned stepsizes are preserved by the following result.
\begin{thm}
\label{thm:1}[Unifying self-tuned stepsize rules]
Consider problem \eqref{eqn:problem1} \fy{with $l=1$}. Let the set $\mathscr B$ be convex, closed, and bounded. Suppose either of the following cases holds: \\
\textbf{case (1)}: $F$ is non-differentiable and Assumption \ref{ass:3} holds for some unknown $C>0$. \\
\textbf{case (2)}: $F$ is continuously differentiable over $\mathscr B$ for all $\xi$, but $\nabla F$ is not Lipschitz over $\mathscr B$ and Assumption \ref{ass:33} holds. \\
\textbf{case (3)}: $F$ is differentiable over $\mathscr B$, it has Lipschitz gradients with an unknown parameter $\mathfrak L_F$, and Assumption \ref{ass:33} holds.\\
In \fy{cases (1) and (2)}, let $\{\beta_t\}$ be generated by algorithm~\eqref{alg:BSSMD}. In \fy{case (3),} let $\{\beta_t\}$ be generated by algorithm~\eqref{alg:BSGMD}. In all cases, let the stepsize $\eta_t$ be given by $\eta_{t}:=\eta_{t-1}\left(1-({\mu_F}/{\mathfrak L_{\omega})}\eta_{t-1}\right)$ for all $t\geq1$, where $0<\eta_0\leq{\mathfrak L_{\omega}}/{2\mu_F}$ is an arbitrary constant.
%(a) If all the problem parameters (including $\mu_F$) are unknown, there exists a threshold $\bar{\eta}$ such that for any $\eta_0<\bar{\eta}$, the stepsize sequence $\eta_{t}^*=\eta_{t-1}^*(1-c\eta_{t-1}^*)~~~\forall t\geq1$ with scalar $c>0$ minimizes an upper bound of the expected error of the %algorithm~(\ref{eqn:SSMD}).
Then: (i) $\{\beta_t\}$ converges to $\beta^*$ a.s., and (ii) there exists a threshold $\bar{\eta}\leq{\mathfrak L_{\omega}}/{2\mu_F}$ such that for any $\eta_0\leq\bar{\eta}$, an upper bound of the error $\EXP{D_\omega(\beta_{t},\beta^*)}$ is minimized for all $(\eta_0, \ldots,\eta_{t-1}) \in \left(0,{\mathfrak L_{\omega}}/{2\mu_F}\right]^t$.
\end{thm}
\begin{proof}
First, we show (i) and (ii) hold in case (1). Let $C_{min}$ denote the minimum of all constants $C>0$ that satisfy Assumption \ref{ass:3} (note that such a constant always exits). Let $\bar {C}\triangleq\max\left\{C_{min},\sqrt{{8M^2\mu_\omega\mu_F^2}/{\mathfrak L_\omega}}\right\}$ and define $\bar \eta\triangleq {4\mu_F\mu_\omega M^2}/{\bar {C}^2}$. Note that $\bar \eta \leq \frac{\mathfrak L_\omega}{2\mu_F}$ from definition of $\bar C$. Let $0<\eta_0 \leq\bar \eta$ be an arbitrary scalar and define $C_0\triangleq \bar C\sqrt{{\bar \eta}/{\eta_0}}$. Note that since $C_0 \geq\bar C\geq C_{min}$, $C_0$ satisfies Assumption \ref{ass:3}. Also, $C_0^2 {\mathfrak L_\omega} \geq {8M^2\mu_\omega\mu_F^2}$. Therefore, for $\eta_0={4\mu_F\mu_\omega M^2}/{{C_0}^2}$, we found a $C_0$ such that all conditions of Proposition \ref{pre:3} are met. Then we can apply Proposition \ref{pre:3} which implies that (i) and (ii) hold. Next, consider case (2). Note that since $f$ is continuously differentiable, the set $\partial f(\beta,\xi)$ is a singleton, i.e., $\partial f(\beta,\xi)=\{\nabla f(\beta,\xi)\}$. From compactness of $\mathscr B$ and continuity of $\nabla f(\cdot,\xi)$, we conclude that Assumption \ref{ass:3} holds for some $C>0$. Next, in a similar fashion to the proof of case (1), we can conclude that (i) and (ii) hold in case (2). The proof for case (3) can be done by invoking Proposition \ref{pre:2} similar to the proof for case (1).
\end{proof}
\begin{rem}
%Theorem \ref{thm:1} provides a practical stepsize rule for the stochastic mirror descent method that can be applied to both smooth and nonsmooth problems even when problem parameters are unknown. 
The unifying stepsize rule \fy{minimizes \ap{the} mean squared error} even when problem parameters are unknown. This suggests that self-tuned stepsizes are robust with respect to the choice of the initial stepsize. We will demonstrate this property of the self-tuned stepsizes in our numerical experiments in Section \ref{sec:num}. This can be seen as an important advantage in contrast with the classical harmonic stepsizes of the form $(\frac{a}{t+b})^c$ that have been seen very sensitive to the choice of three parameters $a,b$ and $c$ (cf. \cite{Spall03}, \fy{Ch. 4}).
\end{rem}
\section{Experimental Results}
\label{sec:num} 
In this section, we analyze the performance of the self-tuned RBSMD schemes for solving the following soft-margin linear support vector machine problem:
\begin{align}
\label{eq:svm}
\min \quad F(\beta)\triangleq\frac{1}{m}\displaystyle\sum\nolimits_{i=1}^m  L(\langle \beta, \mathbf x_i \rangle, y_i)  + \frac{\lambda}{2} \| \beta\|_2^2~,
\end{align}
where $L(\langle \beta, \mathbf x_i \rangle, y_i)\triangleq max\{0,1-y_i\langle \beta, \mathbf x_i \rangle\}$. SVM is known as an effective classification framework and is applied \ap{to} real-world applications such as text categorization, image classification, etc. \cite{cristianini2000introduction}. We use two binary classification data sets namely RCV1 and Skin. The Reuters Corpus Volume I (RCV1) data set \cite{lewis2004rcv1} is a collection of news-wire stories produced by Reuters journalists from 1996-1997. The articles are categorized into four different classes including Industrial, Economics, Social, and Markets. Here, the samples are documents and the features represent the existence/nonexistence of a given token in an article. We use a subset of the original data set with 199,328 samples and 138,921 features. The goal is to predict whether an article belongs to Markets class or not. Skin segmentation data set classifies each pixel of scan photographs as skin or non-skin texture and is used in face and human detection applications. The goal is identifying the skin-like regions. It consists of 3 features, and 245,057 samples out of which 50,859 are the skin samples and 194,198 are non-skin samples. 
\begin{table}[h]
\caption{Initial stepsize values for RCV1 data set}
\centering
\begin{tabular}{|c|c|c|c|}
\hline
$\lambda$ & $\eta_0[1]$ & $\eta_0[2]=\frac{\mathfrak L_\omega}{10 \mu_F}$ & $\eta_0[3]=\frac{\mathfrak L_\omega}{4 \mu_F}$\\ \hline
0.001 & 0.9 & 100 & 250 \\
0.01 & 0.9 & 10 & 25 \\ 
1 & 0.01 & 0.1 & 0.25 \\ \hline
\end{tabular}
\label{tab:step}
\end{table}
Note that \eqref{eq:svm} is a nonsmooth problem and $F$ is strongly convex with parameter $\lambda$. We compare the unifying self-tuned stepsize rule given by Theorem \ref{thm:1} with harmonic stepsizes of the form $\eta_t=\frac{a}{(t+b)}$ where $a$ and $b$ are scalars \cite{Spall03}. Our goal is to compare the sensitivity of the harmonic stepsize rule with different choices of parameters $a$ and $b$, with that of the unifying self-tuned stepsize rule with different initial stepsizes. Let $l=1$, we set $\omega=\frac{1}{2}\| \beta\|^2_2$ where $\mu_\omega=\mathfrak L_\omega=1$. For any fixed value of $\lambda$, we use three different choices of $\eta_0$, all within the interval $\left(0,{\mathfrak L_\omega}/{2 \mu_F}\right]$ as assumed in Theorem \ref{thm:1}. Initial stepsizes are denoted by $\eta_0[1], \eta_0[2]$, and $\eta_0[3]$ and are selected according to Table \ref{tab:step}.
For each experiment, the algorithm is run for $T=10^4$ iterations. Spall [\hspace{1sp}\cite{Spall03}, pg. 113] considers using $b$ that is about $5-10\%$ of the total number of iterations. Accordingly, we choose $b=\small{0.1T}$ and also $b=\small{0.2T}$ which is observed to be a better selection in some of the preliminary experiments. We select $a=\eta_0b$ in order to start from the same initial stepsize as the self-tuned stepsize.  In addition, we compare our proposed scheme with the harmonic stepsize $\eta_0/t$. 
%It can be seen that the harmonic stepsize's performance varies for different data sets. While in some cases by increasing the tuning parameters $a$ and $b$ its performance improves, in other instances its performance deteriorates
%\begin{multicols}{2}
\begin{table}[h]
\setlength{\tabcolsep}{3pt}
\centering
 \begin{tabular}{c| c  c  c}
 & $\eta_0[1]$ & $\eta_0[2]$ & $\eta_0[3]$ \\ \hline\\
\begin{turn}{90}
$\scriptsize{\lambda=0.001}$
\end{turn}
&
\begin{minipage}{.3\textwidth}
\includegraphics[scale=.33, angle=0]{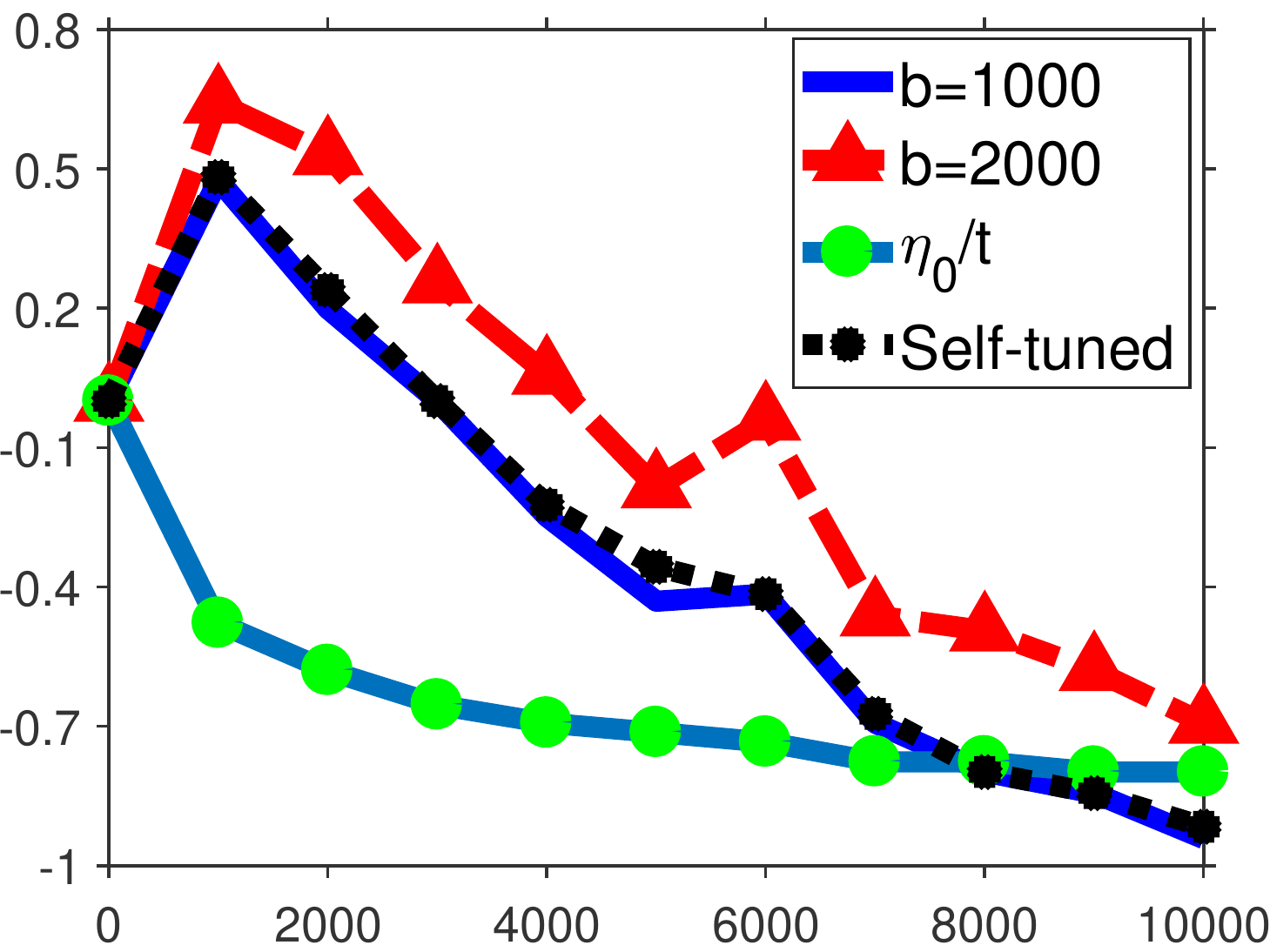}
\vspace{1mm}
\end{minipage}
&
\begin{minipage}{.3\textwidth}
\includegraphics[scale=.33, angle=0]{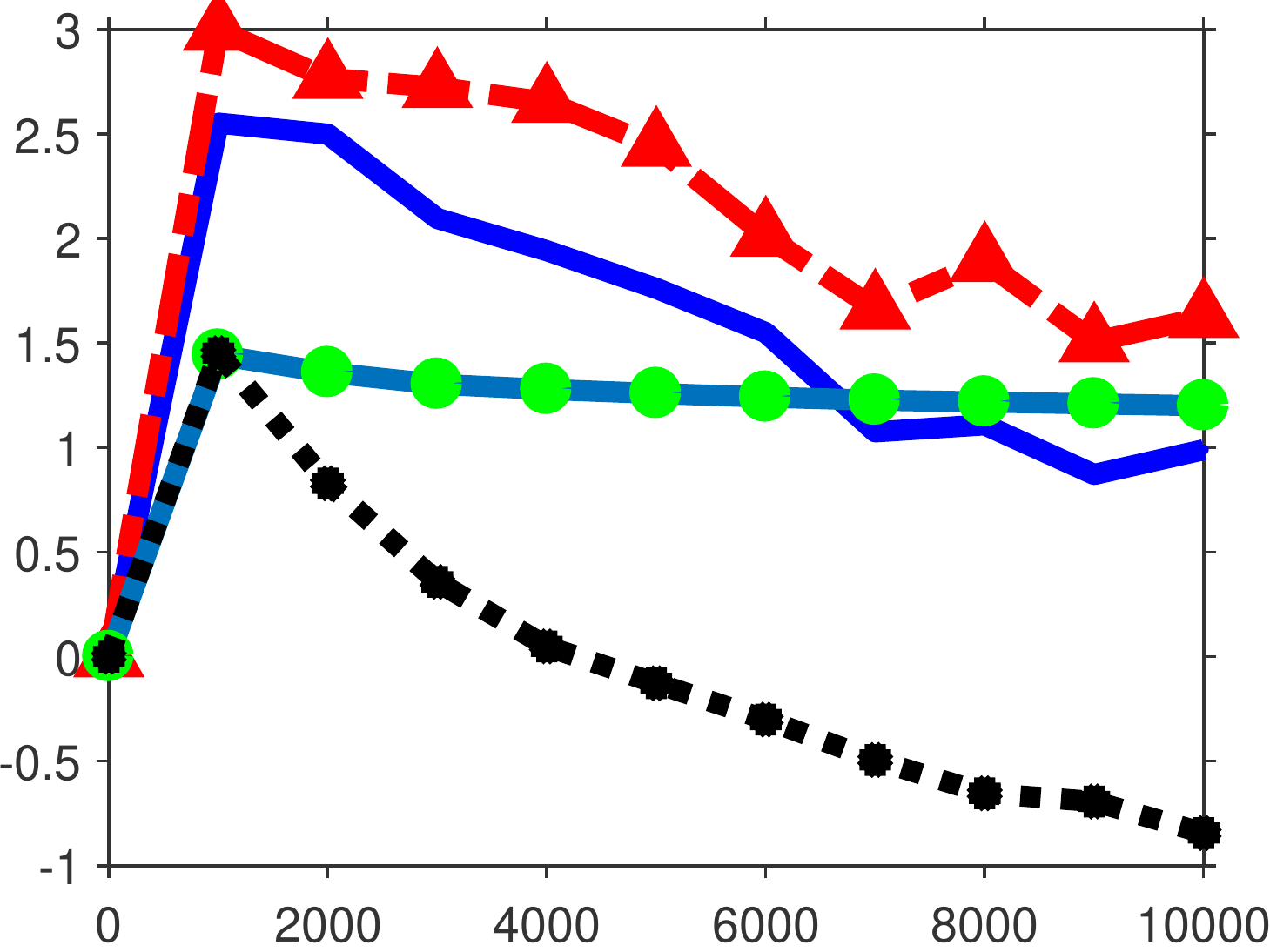}
\vspace{1mm}
\end{minipage}
	&
\begin{minipage}{.3\textwidth}
\includegraphics[scale=.33, angle=0]{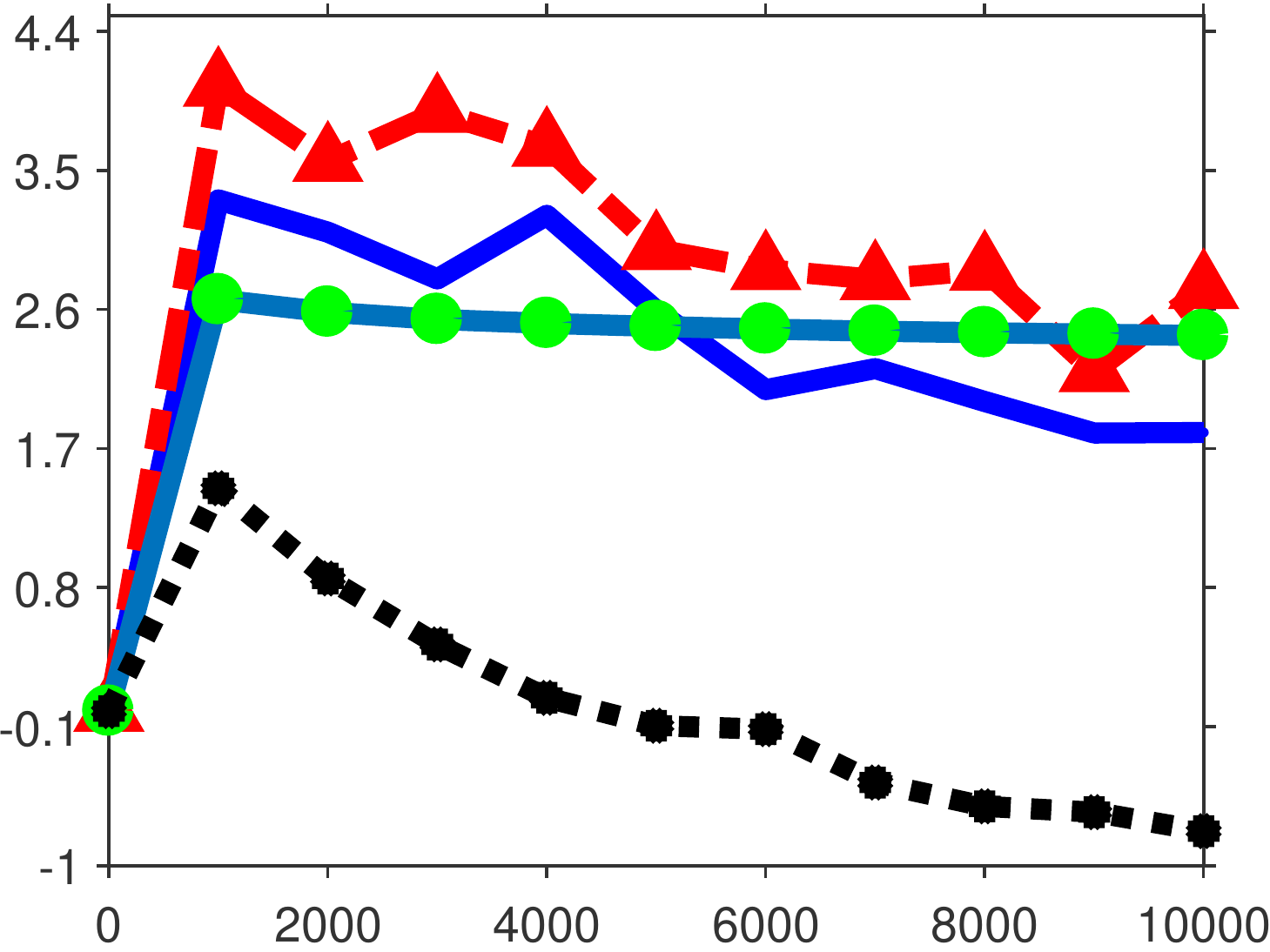}
\vspace{1mm}
\end{minipage}
\\
\begin{turn}{90}
$\scriptsize{\lambda=0.01}$
\end{turn}
&
\begin{minipage}{.3\textwidth}
\includegraphics[scale=.33, angle=0]{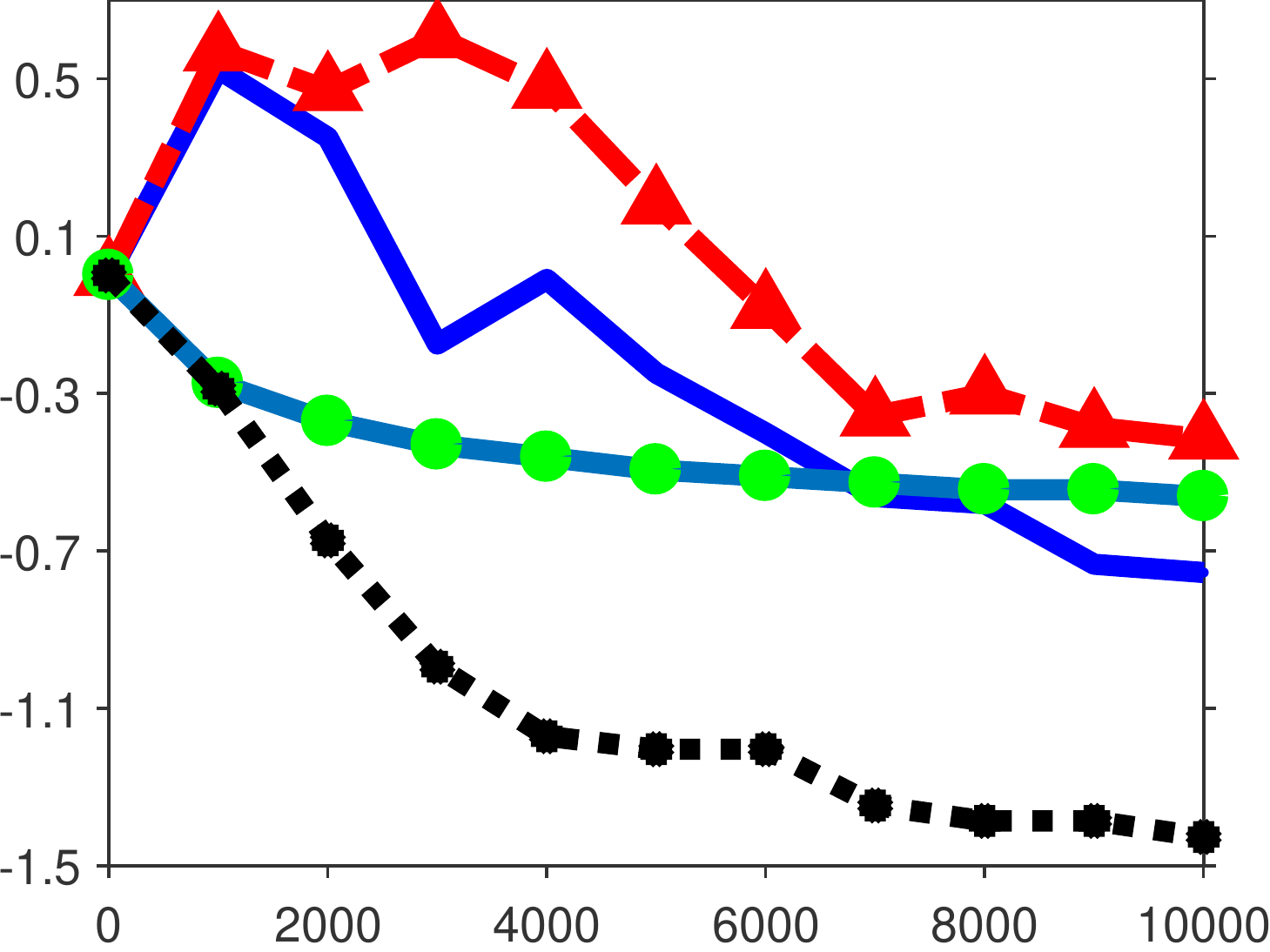}
\vspace{1mm}
\end{minipage}
&
\begin{minipage}{.3\textwidth}
\includegraphics[scale=.33, angle=0]{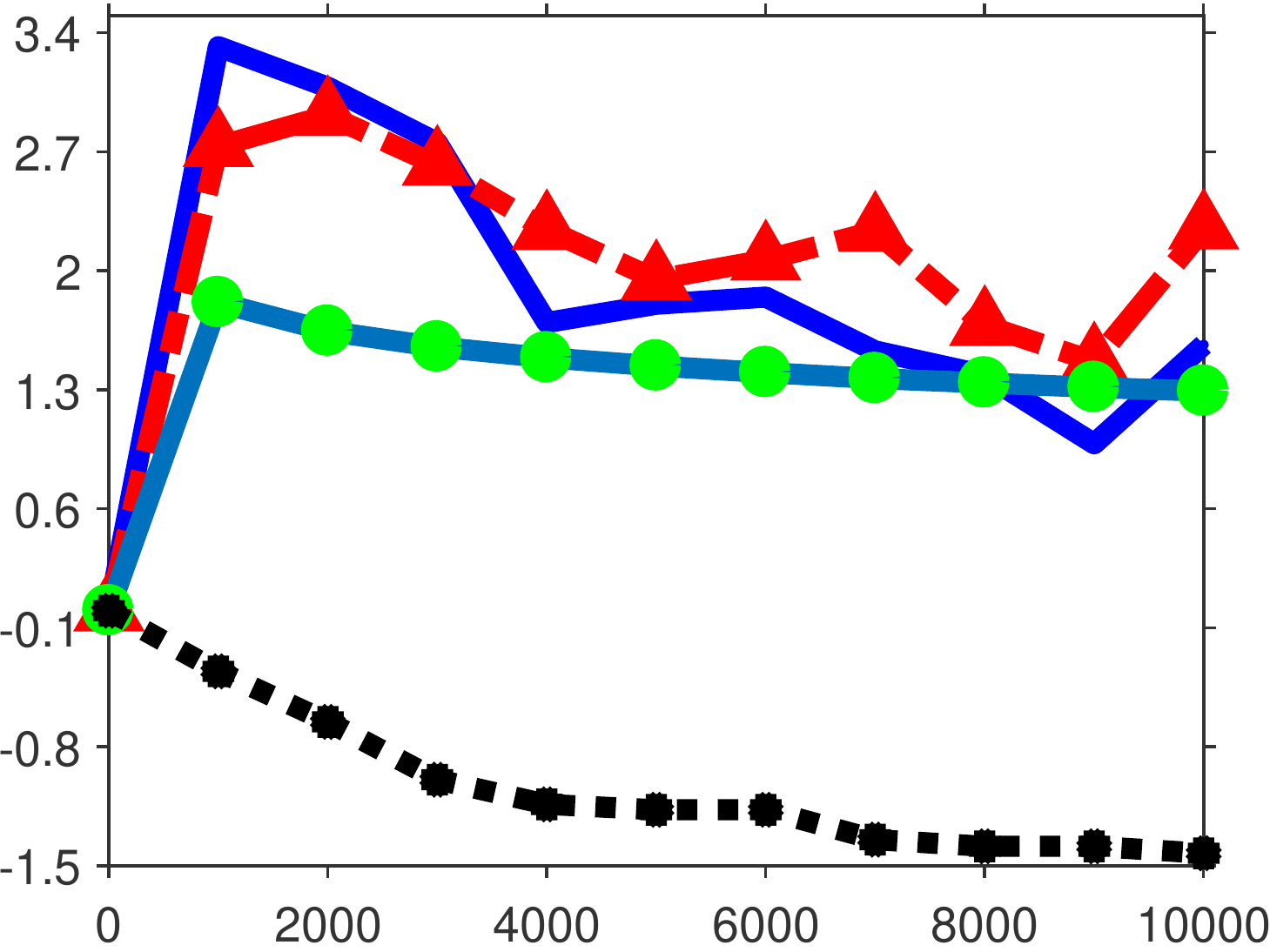}
\vspace{1mm}
\end{minipage}
&
\begin{minipage}{.3\textwidth}
\includegraphics[scale=.33, angle=0]{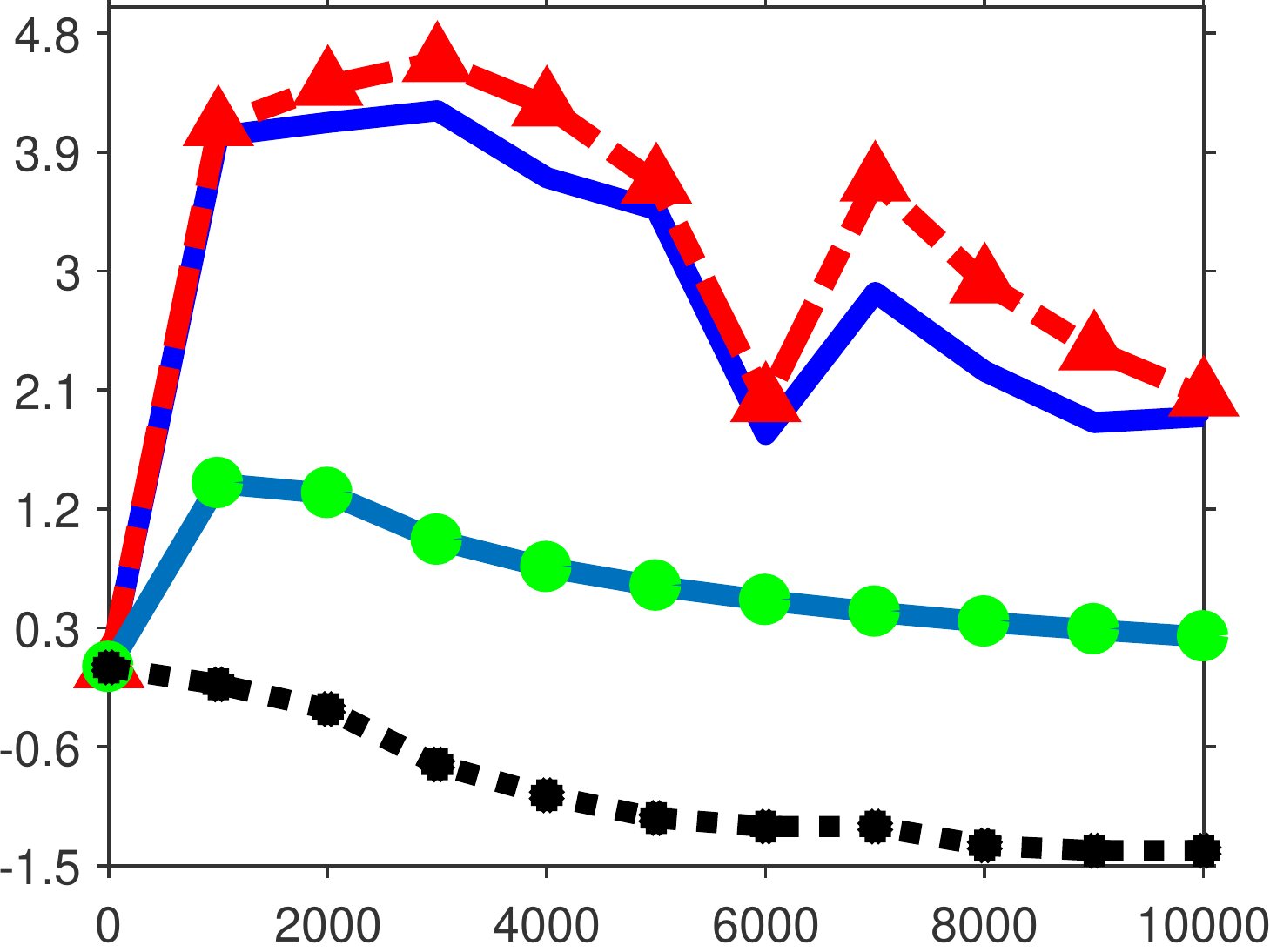}
\vspace{1mm}
\end{minipage}
\\
%0.1
%&
%\begin{minipage}{.3\textwidth}
%\includegraphics[scale=.40, angle=0]{fig27.pdf}
%\end{minipage}
%&
%\begin{minipage}{.3\textwidth}
%\includegraphics[scale=.40, angle=0]{fig28.pdf}
%\end{minipage}
%&
%\begin{minipage}{.3\textwidth}
%\includegraphics[scale=.40, angle=0]{fig29.pdf}
%\end{minipage}
%\\
\begin{turn}{90}
$\scriptsize{\lambda=1}$
\end{turn}
&
\begin{minipage}{.3\textwidth}
\includegraphics[scale=.33, angle=0]{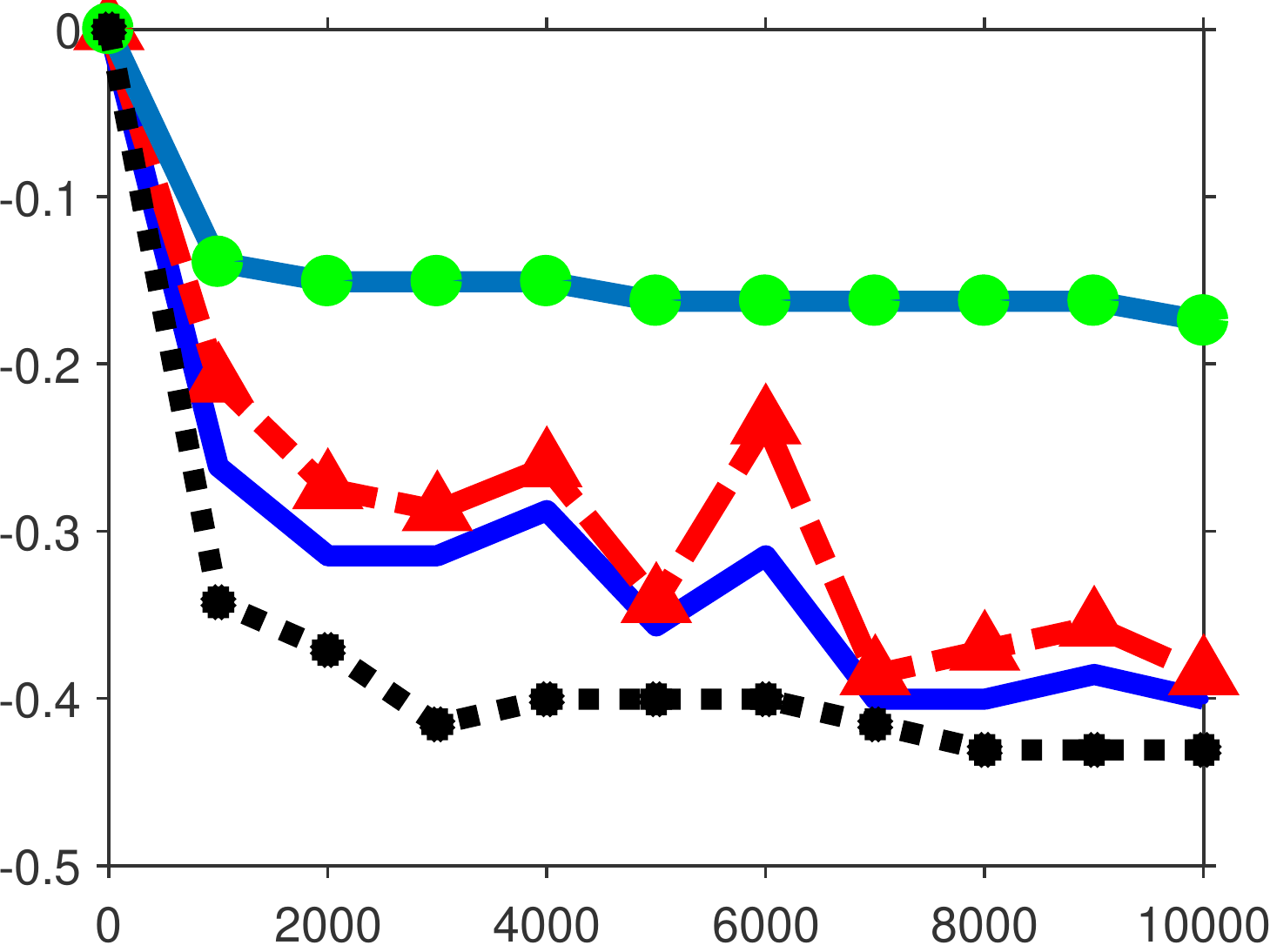}
\end{minipage}
&
\begin{minipage}{.3\textwidth}
\includegraphics[scale=.33, angle=0]{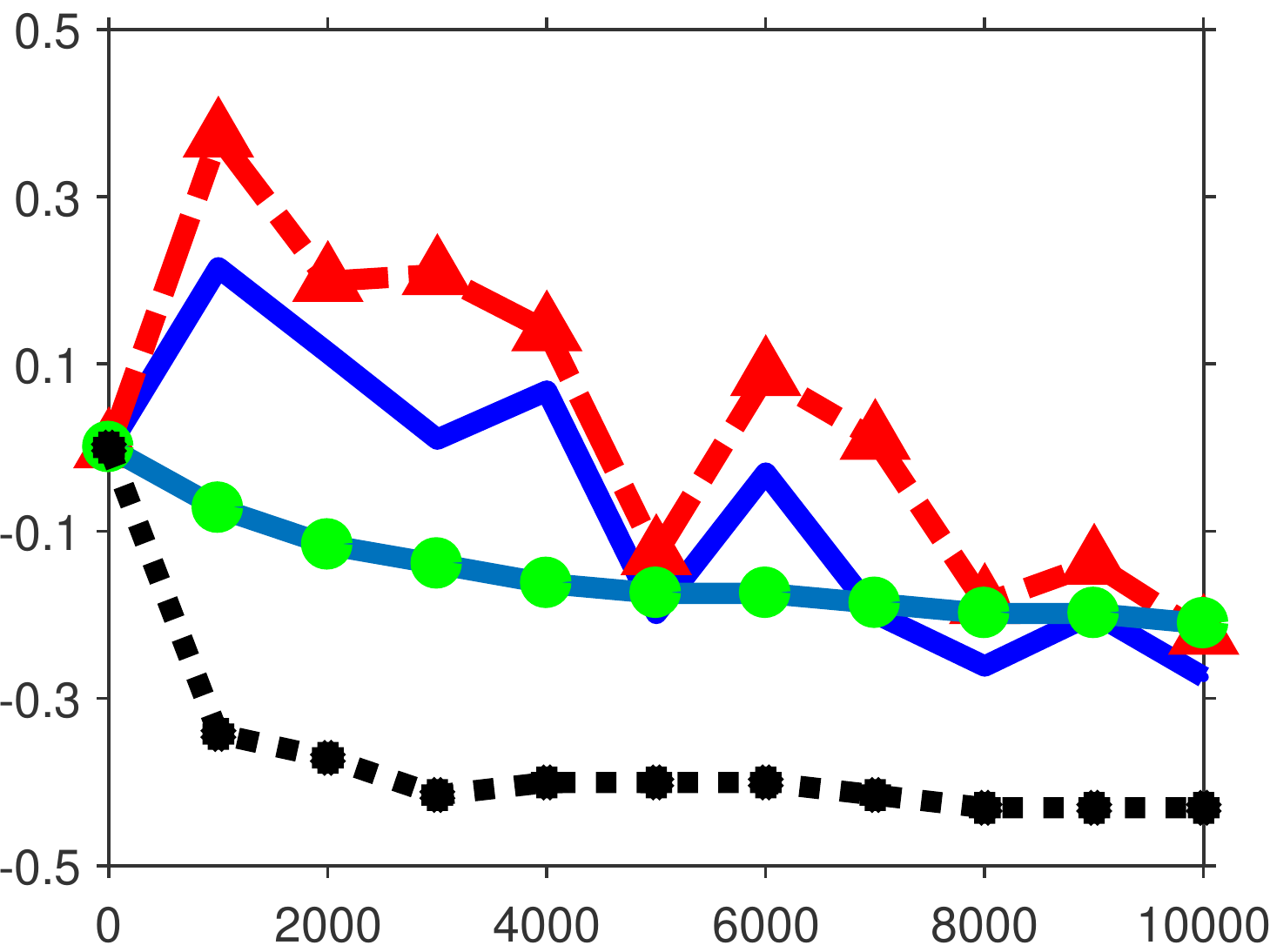}
\end{minipage}
&
\begin{minipage}{.3\textwidth}
\includegraphics[scale=.33, angle=0]{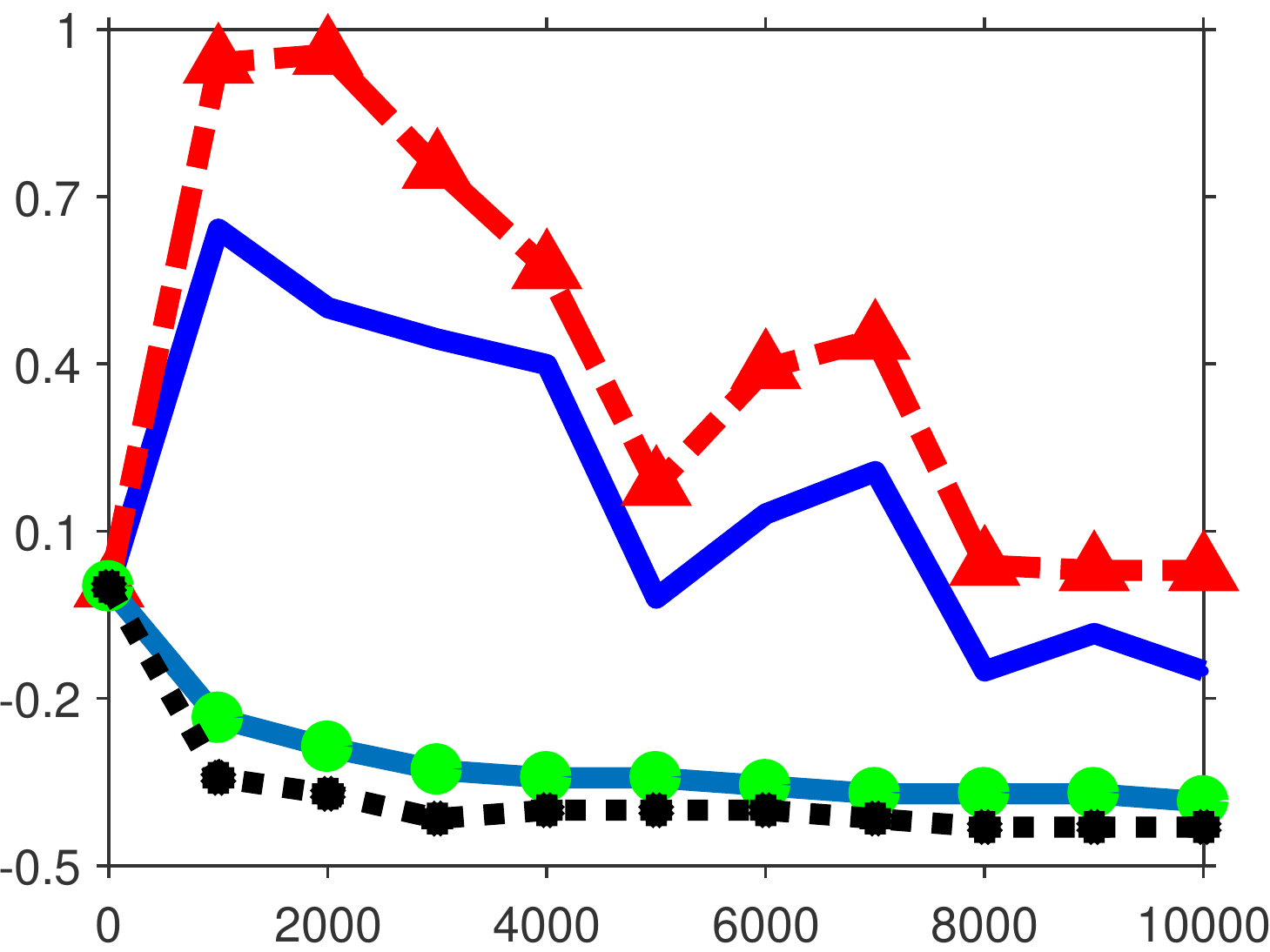}
\end{minipage}
\end{tabular}
\captionof{figure}{RCV1 data set}
\label{fig:RCV}
\end{table}
Figures \ref{fig:RCV} and \ref{fig:skin} demonstrate the performance of these stepsize schemes in terms of logarithm of the averaged objective function $F$. In these plots, the blue and red curves correspond to the harmonic stepsize with parameter $b=1000$ and $b=2000$ respectively, and the green curves denote the stepsize $\eta_0/t$. The black curves \ap{represent} the self-tuned stepsize rule.

We observe in Figures \ref{fig:RCV} and \ref{fig:skin} that the self-tuned stepsize scheme outperforms the harmonic stepsize in most of the experiments. Importantly, the self-tune stepsize is significantly more robust with respect to (i) the choice of $\lambda$; (ii) the data set; and (iii) the initial value of the stepsize. 

\begin{table}[h]
\setlength{\tabcolsep}{3pt}
\centering
 \begin{tabular}{c| c  c  c }
  & $\eta_0=0.00625$ & $\eta_0=0.0125$ & $\eta_0=0.025$ \\ \hline\\
\begin{turn}{90}
$\scriptsize{\lambda=0.001}$
\end{turn}
&
\begin{minipage}{.3\textwidth}
\includegraphics[scale=.40, angle=0]{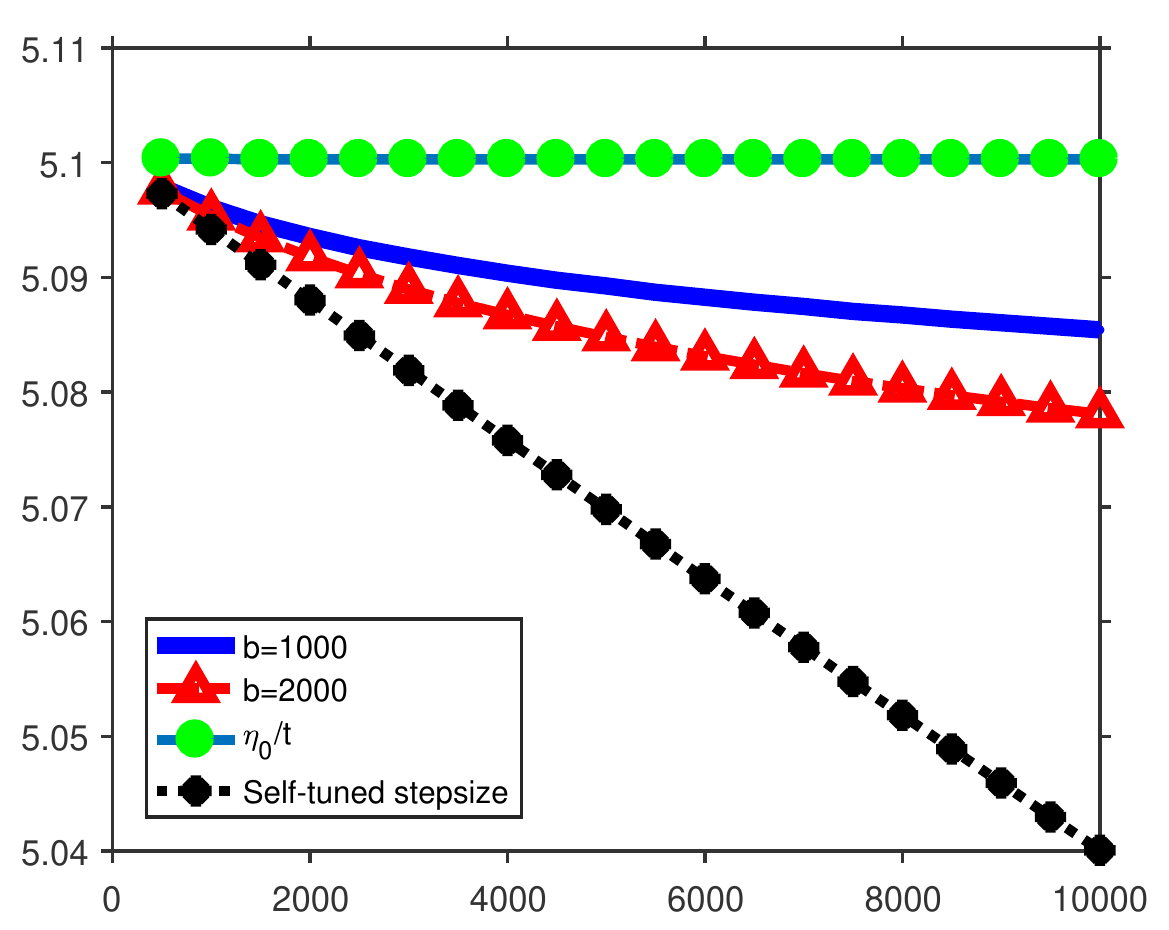}
\end{minipage}
&
\begin{minipage}{.3\textwidth}
\includegraphics[scale=.40, angle=0]{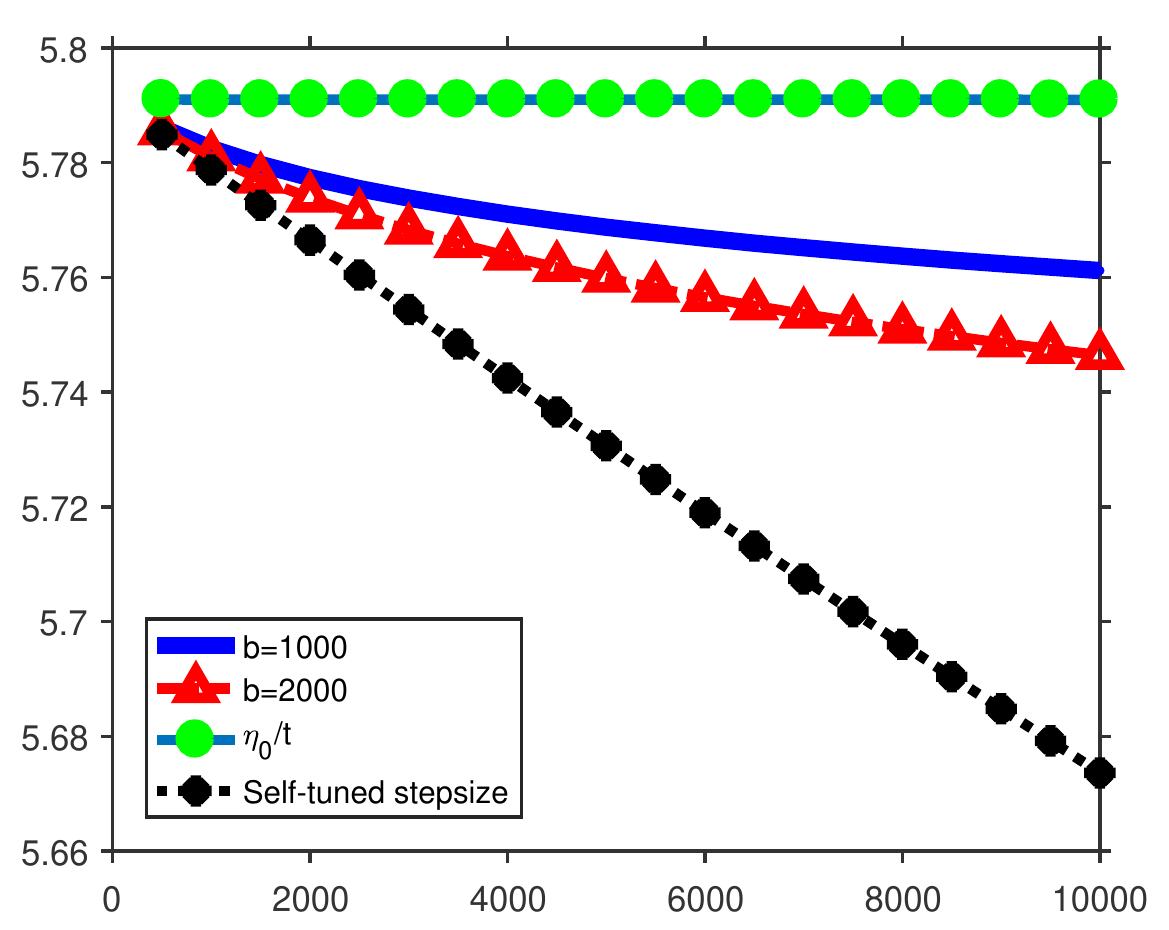}
\end{minipage}
	&
\begin{minipage}{.3\textwidth}
\includegraphics[scale=.40, angle=0]{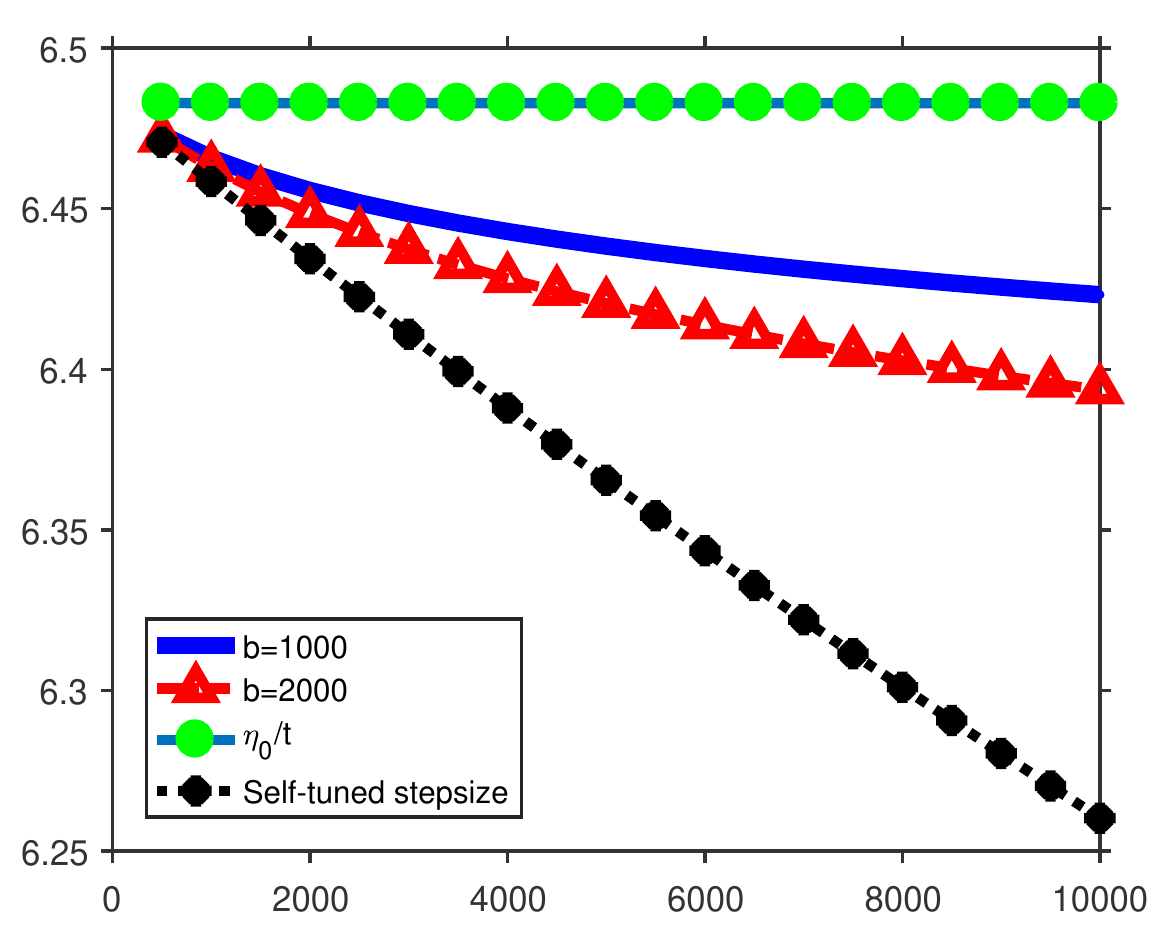}
\end{minipage}
\\
\begin{turn}{90}
$\scriptsize{\lambda=0.01}$
\end{turn}
&
\begin{minipage}{.3\textwidth}
\includegraphics[scale=.40, angle=0]{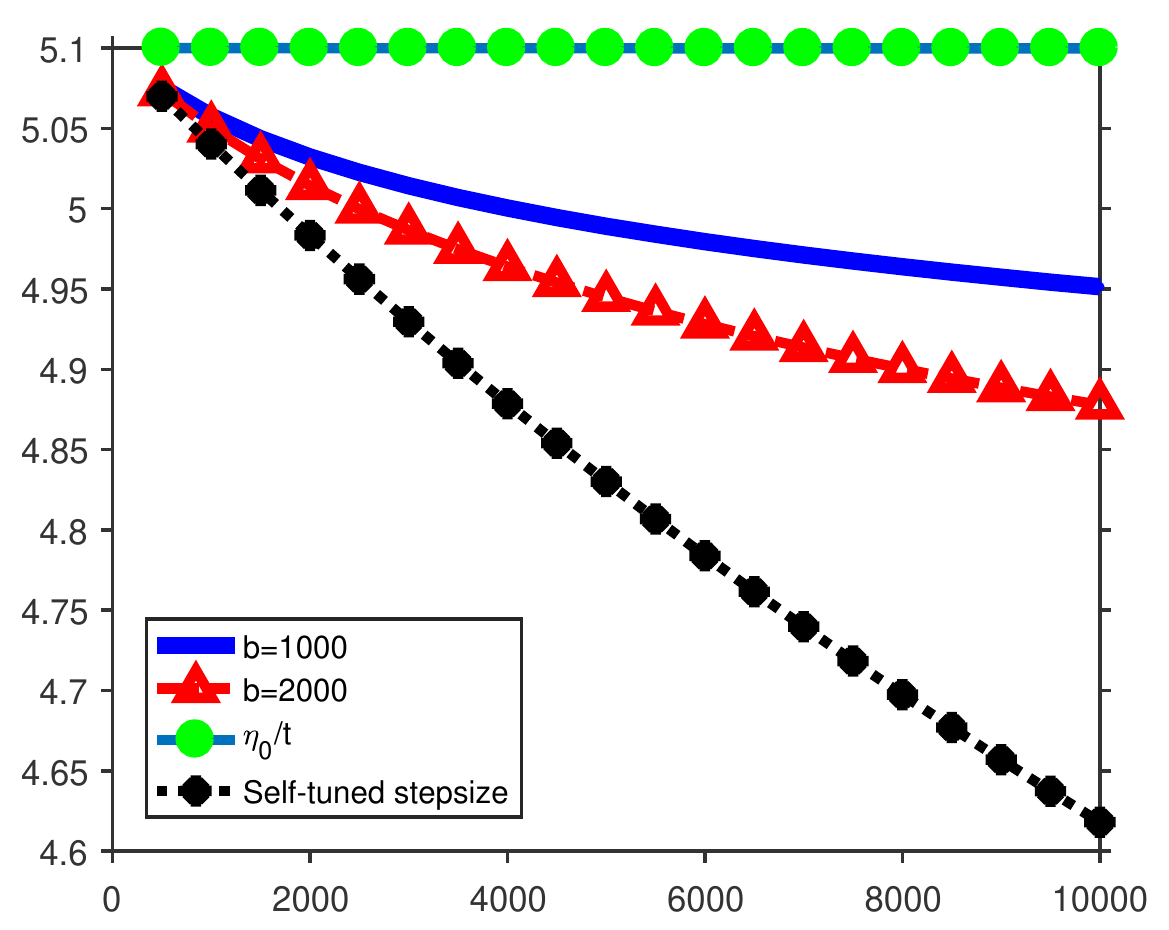}
\end{minipage}
&
\begin{minipage}{.3\textwidth}
\includegraphics[scale=.40, angle=0]{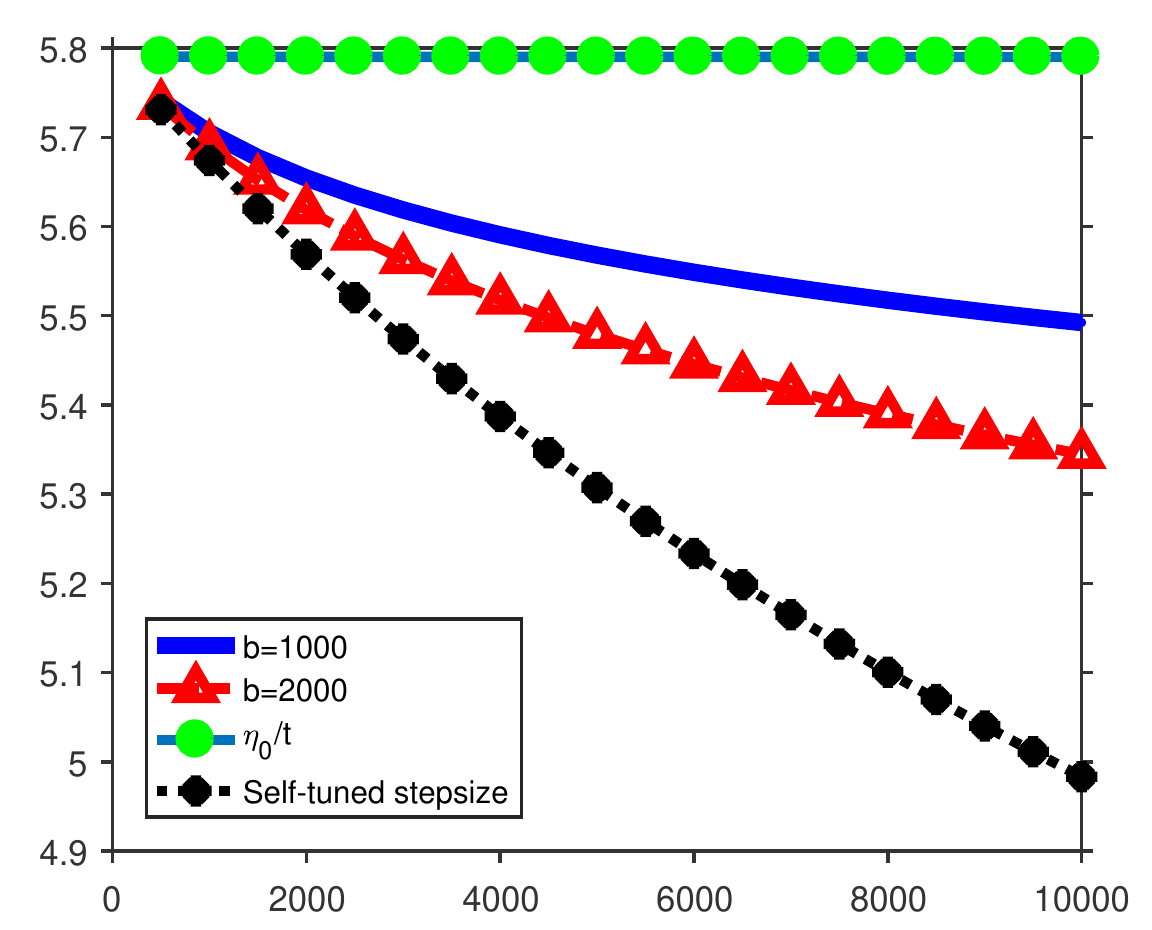}
\end{minipage}
&
\begin{minipage}{.3\textwidth}
\includegraphics[scale=.40, angle=0]{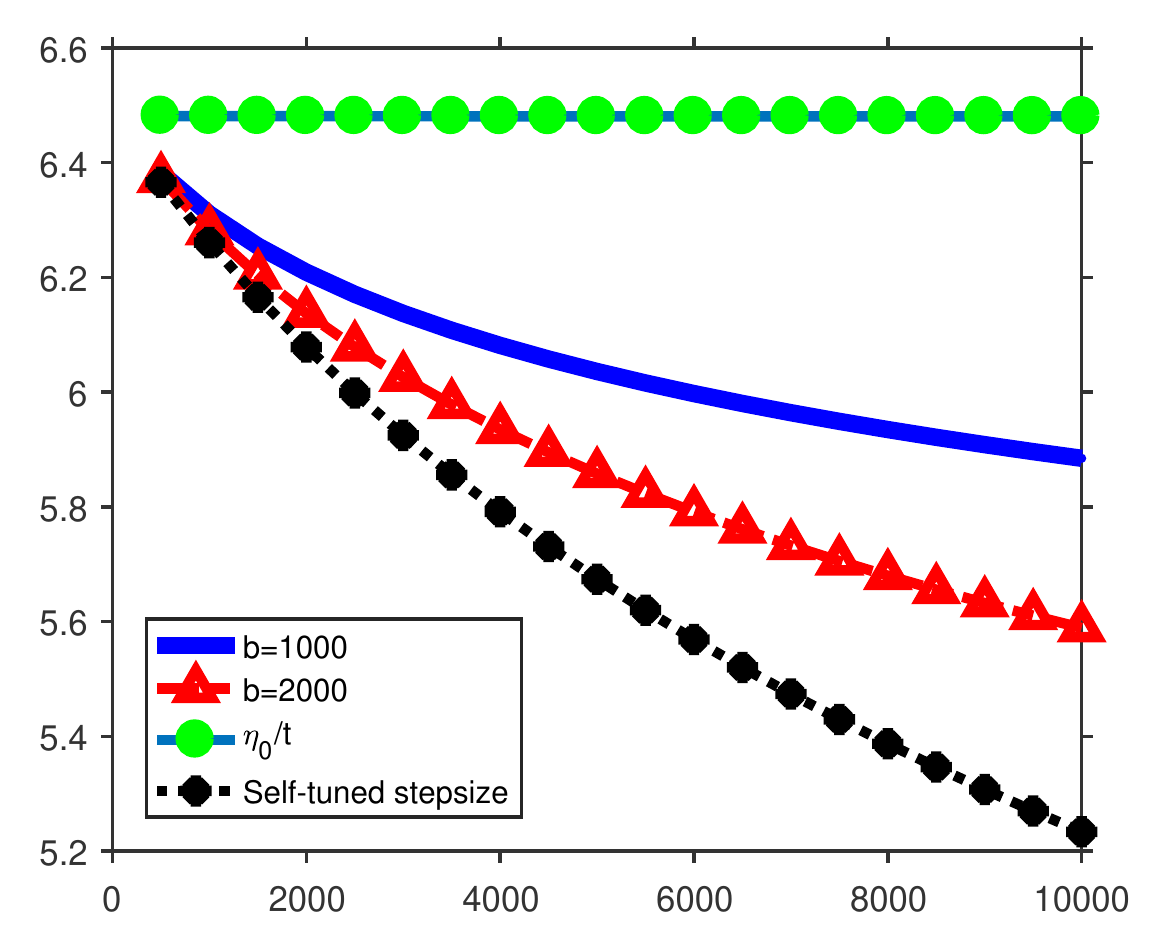}
\end{minipage}
\\
\begin{turn}{90}
$\scriptsize{\lambda=1}$
\end{turn}
&
\begin{minipage}{.3\textwidth}
\includegraphics[scale=.40, angle=0]{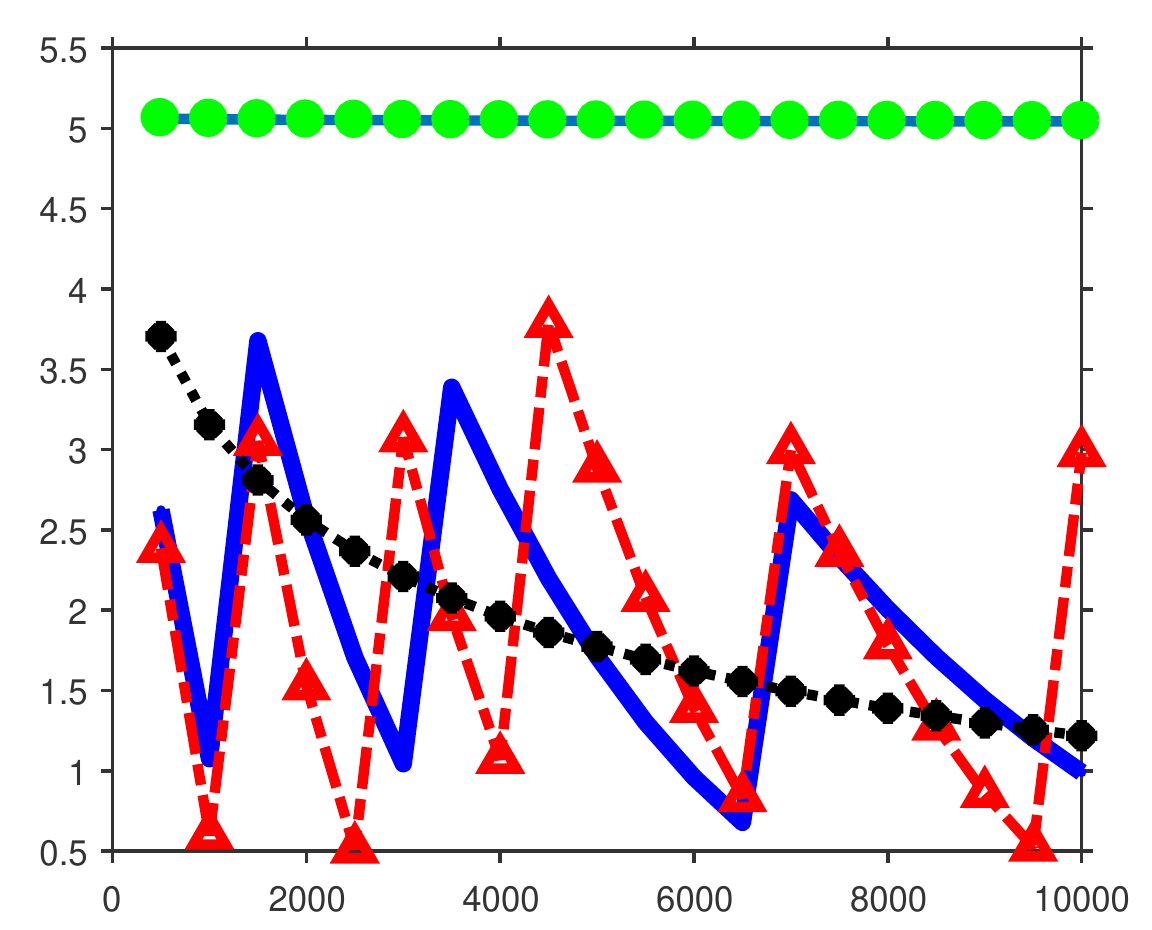}
\end{minipage}
&
\begin{minipage}{.3\textwidth}
\includegraphics[scale=.40, angle=0]{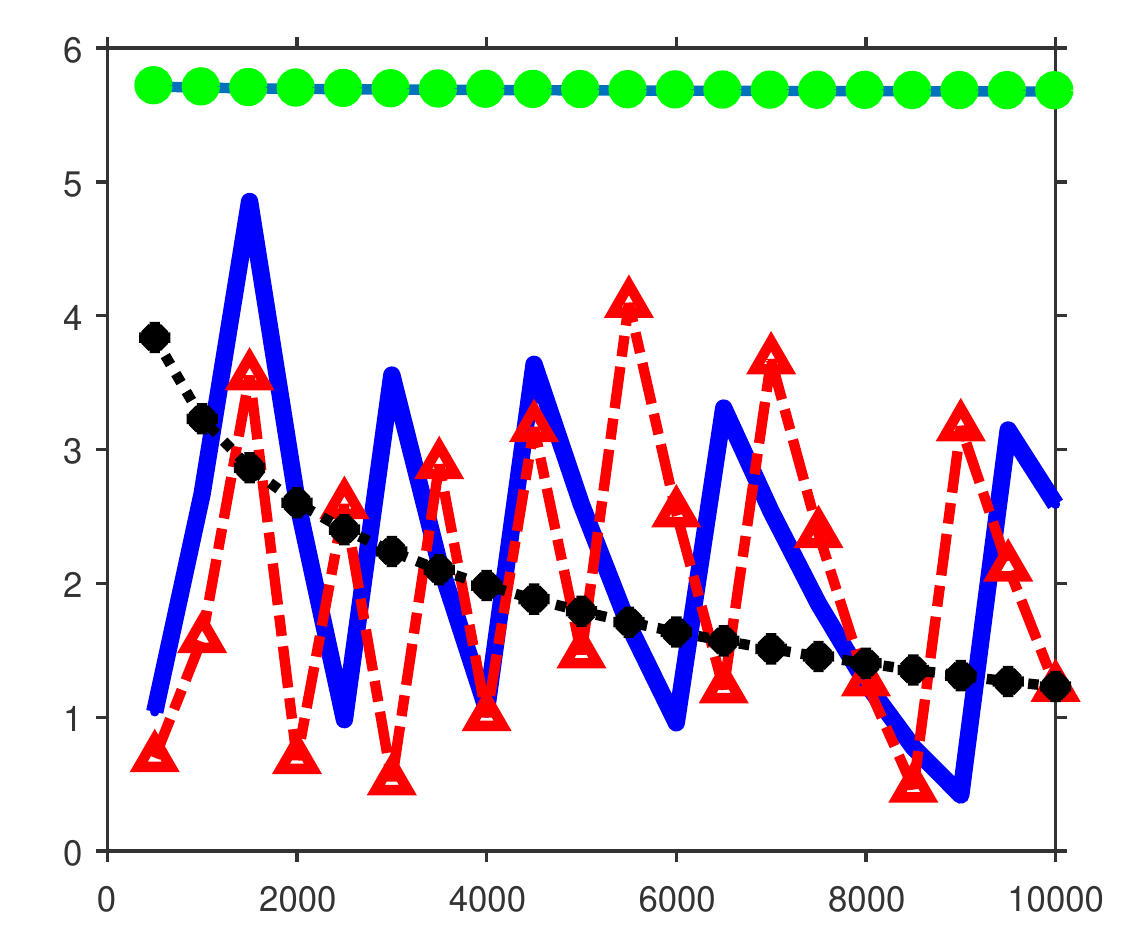}
\end{minipage}
&
\begin{minipage}{.3\textwidth}
\includegraphics[scale=.40, angle=0]{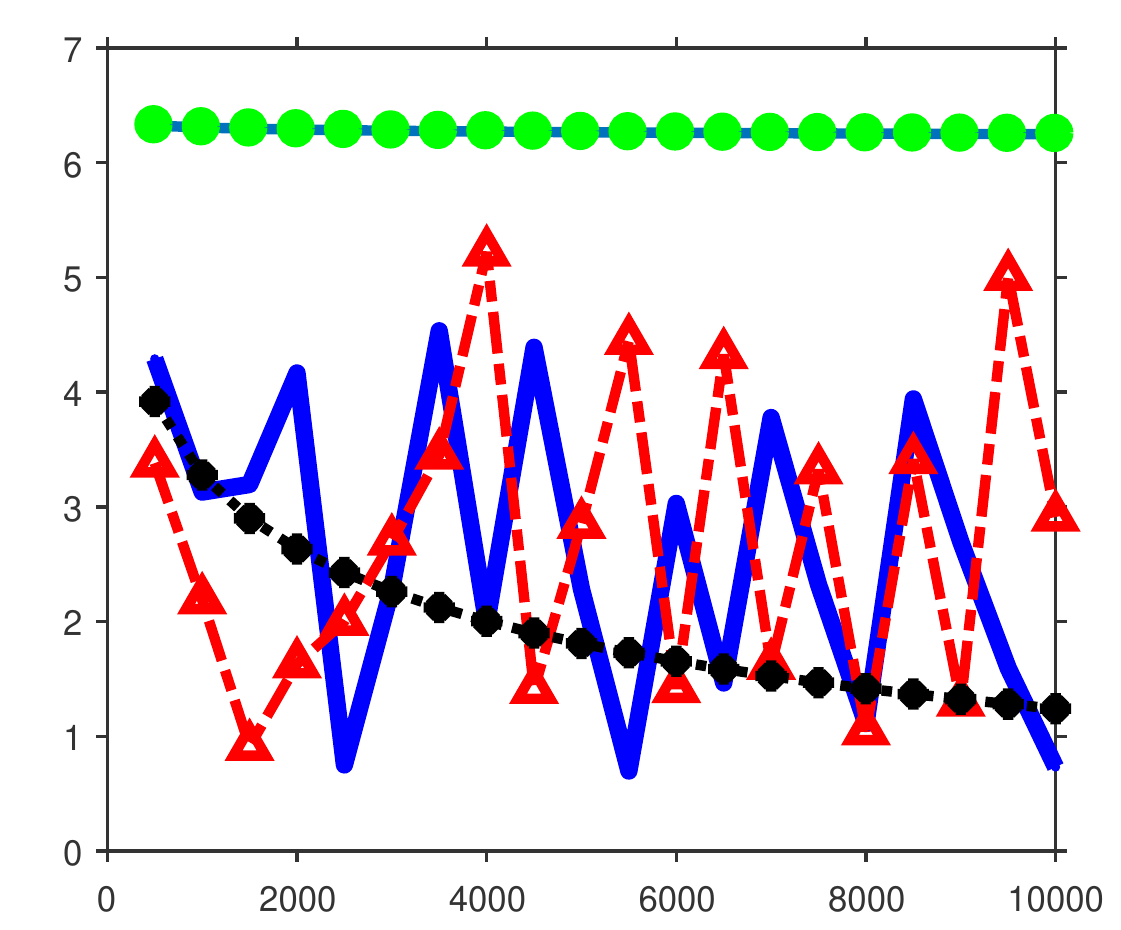}
\end{minipage}
\end{tabular}
\captionof {figure}{Skin data set}
\label{fig:skin}
\end{table}
\section{Concluding Remarks}
\label{sec:conclusion}
We consider stochastic optimization problems with strongly convex objective functions. 
%Much of the past research on SMD methods has focused on convergence and rate analysis in terms of order of the error bounds. However, the stepsize choice plays a key role in the performance of the this class of algorithms.We consider nonsmooth, smooth, and high-dimensional stochastic optimization problems. 
We develop self-tuned stepsize rules for stochastic subgradient and gradient randomized block coordinate mirror descent methods. For each scheme, we prove almost sure convergence and show that under the self-tuned stepsize rules, the error bound of the RBSMD scheme is minimized. In the case that some problem parameters are unknown, we develop a unifying self-tuned update rule for which an error bound of the scheme is minimized for any arbitrary and small enough initial stepsize.
\section {Appendix}
\textbf{Proof of Lemma \ref{lemma:selftunedGeneral}:}
\noindent (a) To show part (a), we first use induction on $t$ to show that $e_t$ satisfies 
\begin{equation}
\label{eq:70}
e_t(\eta_0^*, \eta_1^*,\ldots,\eta_{t-1}^*)=\frac{2\delta}{\theta}\eta_t^*, \quad \hbox{for all }t \geq 0.
\end{equation}
Note that it holds for $t=0$ from the definition $\eta_0^*=\frac{\theta}{2\delta}er_0$. Next, let us assume \eqref{eq:70} holds for $t$. From this and relation
\eqref{eqn:seqer}, we have
\begin{align*}
e_{t+1}(\eta_0^*, \eta_1^*,\ldots,\eta_{t}^*)&=(1-\theta\eta_t^*) e_t(\eta_0^*, \eta_1^*,\ldots,\eta_{t-1}^*)+\delta{\eta_{t}^*}^2 = (1-\theta\eta_t^*)\frac{2\delta}{\theta}\eta_t^*+\delta{\eta_{t}^*}^2
\\&=\frac{2\delta}{\theta}\eta_t^*\Big(1-\theta\eta_t^*+\frac{\theta\eta_t^*}{2} \Big)=\frac{2\delta}{\theta}\eta_t^*\Big( 1- \frac{\theta\eta_t^*}{2} \Big)=\frac{2\delta}{\theta}\eta_{t+1}^*,
\end{align*}
where in the last equation, we used the definition of $\eta_{t+1}^*$. This implies that relation \eqref{eq:70} holds for $t+1$ and therefore, for any $t\geq 0$. We now use induction on $t$ to prove that $(\eta_0^*, \eta_1^*,\ldots,\eta_{t-1}^*)$ minimizes $e_t$ for all $t \geq 1$. By the definition of $er_1$ and the relation $er_1(\eta_0^*)=\frac{2\delta}{\theta}\eta_{1}^*$ shown previously, we have
\begin{equation*}
er_1(\eta_0)-er_1(\eta_0^*)=(1-\theta\eta_0)er_0+\delta\eta_{0}^2  -\frac{2\delta}{\theta}\eta_{1}^*.
\end{equation*}
Therefore, using relation $\eta_{1}^*=\eta_{0}^*\left(1-\frac{\theta}{2}\eta_{0}^*\right)$ and $\eta_0^*=\frac{\theta}{2\delta}er_0$, we can write
\begin{align*}
&er_1(\eta_0)-er_1(\eta_0^*)=(1-\theta\eta_0)\frac{2\delta}{\theta}\eta_0^*+\delta\eta_{0}^2-\frac{2\delta}{\theta}\eta_0^*\Big(1-\frac{\theta}{2}\eta_0^*\Big)=\delta(\eta_0-\eta_0^*)^2.
\end{align*}
This implies that part (a) holds for $t=1$. In the rest of the proof, for the sake of simplicity, we use $e_{t+1}$ for an arbitrary vector $(\eta_0, \eta_1,\ldots, \eta_t) \in \mathbb U_{t+1}$ and $er^*_{t+1}$ for $e_{t+1}$ evaluated at $(\eta_0^*, \eta_1^*,\ldots,\eta_{t}^*)$. 
Now suppose part (a) holds for some $t\geq 1$ implying that $e_t \geq e_t^*$ holds for any  $(\eta_0, \eta_1,\ldots, \eta_{t-1}) \in \mathbb U_t$. Using \eqref{eqn:seqer} and~(\ref{eq:70}), we have 
\begin{eqnarray*}
e_{t+1}-e_{t+1}^*=(1-\theta\eta_t)e_t+\delta\eta_{t}^{2}-\frac{2\delta}{\theta}\eta_{t+1}^*.
\end{eqnarray*}
Using $e_t \geq e_t^*$, relation~(\ref{eq:70}), the definition of $\eta^*_{t+1}$ and that $\eta_t \leq \frac{1}{\theta}$, we get 
\begin{align*}
&e_{t+1}-e_{t+1}^* \geq (1-\theta\eta_t)\frac{2\delta}{\theta}\eta_{t}^*+\delta\eta_{t}^{2}-\frac{2\delta}{\theta}\eta_{t}^*\Big(1-\frac{\theta}{2}\eta_t^*\Big)=\delta(\eta_t-\eta_t^*)^2.
\end{align*}
Therefore, part (a) holds for $t+1$. We conclude that the result of part (a) is true for any $t\geq 1$.

\noindent (b) Using the recursive relation $\eta_{t+1}^*=\eta_{t}^*\left(1-\frac{\theta}{2}\eta_{t}^*\right)$, we have
\begin{equation*}
\frac{1}{\eta_{t+1}}=\frac{1}{\eta_t\left(1-\frac{\theta}{2}\eta_t\right)}=\frac{1}{\eta_t}+\frac{\frac{\theta}{2}}{1-\frac{\theta}{2}\eta_t}, \quad \hbox{for all } t\geq 0.
\end{equation*}
Summing up from $t=0$ to $k$ and canceling the common terms from both sides, we obtain
\begin{equation}
\label{eq:94}
\frac{1}{\eta_{k+1}}=\frac{1}{\eta_0}+\frac{\theta}{2} \sum_{t=0}^k\frac{1}{1-\frac{\theta}{2}\eta_t} > \frac{\theta}{2}\sum_{t=0}^k\frac{1}{1-\frac{\theta}{2}\eta_t}.
\end{equation}
Note that from the definition of $\eta_0^*$ and $er_0$, we have $0<\eta_0^*\leq\frac{1}{\theta}$. From relation $\eta_{t}^*=\eta_{t-1}^*\left(1-\frac{\theta}{2}\eta_{t-1}^*\right)$ we have $0<\eta_t^*\leq\frac{1}{\theta}$ for all $t\geq 0$. Consequently, the term $1-\frac{\theta}{2}\eta_t^*$ is a number between zero  and one. Therefore, $\left(1-\frac{\theta}{2}\eta_t^*\right)^{-1}>1$ which implies that $
\sum_{t=0}^k\left(1-\frac{\theta}{2}\eta_t^*\right)^{-1}>k+1$. Therefore, using relation~(\ref{eq:94}), for all $k\geq 1$ we have 
 $\displaystyle\eta_k^*<\frac{2}{\theta k}$.
Combining inequality~(\ref{eq:70}) and the preceding inequality, we obtain the desired result.\\
\noindent (c) First, we show $\sum_{t=0}^\infty \eta_t^*=\infty$. From $\eta_{t}^*=\eta_{t-1}^*\left(1-\frac{\theta}{2}\eta_{t-1}^*\right)$ for all $t\geq 0$, we obtain
\begin{equation}
\label{eq:71}
\eta_{t+1}^*=\eta_0^*\prod_{i=0}^{t}\left(1-\frac{\theta}{2}\eta_i^*\right).
\end{equation}
Note that since $\eta_0^* \in \left(0,\frac{1}{\theta}\right]$, from $\eta_{t}^*=\eta_{t-1}^*\left(1-\frac{\theta}{2}\eta_{t-1}^*\right)$ it follows that $\{\eta_t^*\}$ is positive non-increasing sequence. Therefore, the limit $\lim_{t \to \infty}\eta_t^*$ exists and it is less than $\frac{2}{\theta}$. Thus, by taking the limits from both sides in $\eta_{t}^*=\eta_{t-1}^*\left(1-\frac{\theta}{2}\eta_{t-1}^*\right)$, we obtain $\lim_{t \to \infty}\eta_t^*=0$. Then, by taking limits in~(\ref{eq:71}), we further obtain 
\begin{equation*}
\lim_{t \to \infty}\prod_{i=0}^{t}\left(1-\frac{\theta}{2}\eta_i^*\right)=0.
\end{equation*}
To arrive at a contradiction, suppose that  $\sum_{i=0}^\infty \eta_i^*<\infty$. Then, there is an $\epsilon \in (0,1)$ such that for $j$ sufficiently large, we have $
\frac{\theta}{2}\sum_{i=j}^{t}\eta_i^*\leq \epsilon, \hbox{ for all } t\geq j. $
Since $\prod_{i=j}^{t}\left(1-\frac{\theta}{2}\eta_i^*\right)\geq 1-\frac{\theta}{2}\sum_{i=j}^{t}\eta_i^*$ for all $j<t$, by letting $t \to \infty$, we obtain for all $j$ sufficiently large, 
\begin{equation*}
\prod_{i=j}^{\infty}\left(1-\frac{\theta}{2}\eta_i^*\right)\geq 1-\frac{\theta}{2}\sum_{i=j}^{\infty}\eta_i^* \geq1-\epsilon >0.
\end{equation*}
This contradicts the statement $\lim_{t \to \infty}\prod_{i=0}^{t}\left(1-\frac{\theta}{2}\eta_i^*\right)=0$. Hence, we conclude that $\sum_{t=0}^\infty \eta_t^*=\infty$. Next, we show that $\sum_{t=0}^\infty {\eta_t^*}^2< \infty$. From $\eta_{t}^*=\eta_{t-1}^*\left(1-\frac{\theta}{2}\eta_{t-1}^*\right)$ we have
\begin{equation*}
\eta^*_i=\eta^*_{i-1}-\frac{\theta}{2}{\eta^{*}_{i-1}}^2,\quad  \hbox{for all } i\geq 1.
\end{equation*}
Summing the preceding relation from $i=0$ to $t$ and canceling the common terms, we obtain
\begin{equation*}
\eta^*_t=\eta^*_{0}-\frac{\theta}{2}\sum_{i=0}^{t-1}\eta^{*2}_{i},\quad  \hbox{for all } t\geq 1.
\end{equation*}
By taking limits and recalling that $\lim_{t \to \infty}\eta^{*}_t=0$, we obtain the desired result.
\bibliographystyle{IEEEtran}
\bibliography{reference} 
\end{document}